\documentclass[12pt]{scrartcl}

\usepackage{url}

\usepackage{mathtools}  
\usepackage{amsmath}
\usepackage{amsthm}
\usepackage{amsxtra}

\usepackage{amssymb}

\usepackage[all]{xy}

\usepackage[english]{hyperref}
\usepackage[english]{babel}
\usepackage[utf8]{inputenc}

\newcommand*{\field}[1]{\mathbb{#1}}
\newcommand*{\Q}{\field{Q}}
\newcommand*{\F}{\field{F}}
\newcommand*{\R}{\field{R}}
\newcommand*{\C}{\field{C}}
\newcommand*{\Z}{\field{Z}} 

\newcommand*{\OF}{\mathcal{O}_\F}
\newcommand*{\DF}{\mathcal{D}_\F}

\newcommand*{\group}[1]{\mathrm{#1}}

\newcommand*{\SU}{\group{SU}}
\newcommand*{\Ug}{\group{U}}
\newcommand*{\Og}{\group{O}}
\newcommand*{\SO}{\group{SO}}
\newcommand*{\SL}{\group{SL}}

\newcommand*{\Mp}{\group{Mp}}

\newcommand*{\GammaU}[1]{\Gamma_{#1}^{\Ug}}
\newcommand*{\GammaO}[1]{\Gamma_{#1}^{\Og}}

\newcommand*{\VF}{V_{\F}}
\newcommand*{\VQ}{V_{\Q}}
\newcommand*{\VFC}{V_{\F}(\C)} %*
\newcommand*{\VQR}{V_{\Q}(\R)} %*
\newcommand*{\VQC}{V_{\Q}(\C)} %*
\newcommand*{\PFR}{\mathbb{P}^1{(V_{\F})(\C)}} %*
\newcommand*{\PQC}{\mathbb{P}^1{(V_{\Q})(\C)}}  %*

\newcommand*{\GrU}{\mathrm{Gr}_\Ug} 
\newcommand*{\GrO}{\mathrm{Gr}_\Og} 
\newcommand*{\GrK}{\mathcal{G}_K}

\newcommand*{\HU}{\mathcal{H}_\Ug}
\newcommand*{\HO}{\mathcal{H}_\Og}
\newcommand*{\Hp}{\mathbb{H}}

\newcommand*{\coneU}{\mathcal{K}_\Ug}
\newcommand*{\coneO}{\mathcal{K}_\Og} 

\newcommand*{\quadN}{\mathcal{N}} % Null quadric

\newcommand*{\UXmv}[1]{X_{#1}} % modular variety for group #1
\newcommand*{\UXGamma}{\UXmv{\Gamma}}
\newcommand*{\UXGammaBB}{\UXmv{ \Gamma, \mathcal{BB}}^*}

\DeclarePairedDelimiter{\hlfa}{\langle}{\rangle}
\newcommand*{\hlf}[2]{\hlfa*{ #1, #2}} % hermitian form
\newcommand{\hlfempty}{\hlf{\cdot}{\cdot}}
\DeclarePairedDelimiter{\blfp}{(}{)}
\newcommand*{\blf}[2]{\blfp*{ #1, #2}} % bilinear form
\newcommand*{\blfempty}{\blf{\cdot}{\cdot}} 

\newcommand*{\Qf}[1]{q\!\left( #1 \right)} % quadratic form
\newcommand*{\QfNop}{q} % 'No parenthesis
\DeclarePairedDelimiter{\abs}{\lvert}{\rvert}

\DeclarePairedDelimiter{\Wposd}{(}{)} % Weyl chamber positivity, delimiter
\newcommand*{\Wpos}[2]{\Wposd*{ #1, #2}} % For Weyl chamber positivity condition

\DeclareMathOperator{\tr}{Tr}

\newcommand*{\Mweak}{\mathcal{M}^!} % weakly holomorphic mf
\newcommand*{\ebase}{\mathfrak{e}} % std. basis for group algebra

\newcommand*{\widthell}{N_\ell} %  Width of cusp

\DeclareMathOperator{\Div}{div}

\newcommand*{\HeegO}{H}
\newcommand*{\HeegU}{\mathbf{H}}

\newcommand*{\Res}[2]{\left\lbrace #1, #2 \right\rbrace} % residue pairing 

\newcommand*{\Chow}{\mathrm{CH}^1} % first Chow group
\newcommand*{\Abdl}[1]{\mathcal{L}_{#1}} % bundle of automorphic forms
\newcommand*{\BDiv}{\mathcal{B}} % group of boundary divisors

\newtheorem{theorem}{Theorem}[section]
\newtheorem{definition}{Definition}[section]
\newtheorem{prop}{Proposition}[section]
\newtheorem{lemma}{Lemma}[section]
\newtheorem{assume}{Assumption}[section]
\newtheorem{rmk}{Remark}[section]
\newtheorem*{rmk*}{Remark}
\newtheorem{example}{Example}[section]
\newtheorem{corollary}{Corollary}[section]

\newcommand*{\acknowledegements}{\textit{Acknowledegements:}\ }

\title{Borcherds Products on Unitary Groups}
 \author{Eric Hofmann%\\
\footnote{%
 Eric Hofmann. Mathematisches Institut Universität Heidelberg,\newline
Im Neuenheimer Feld 288, 69120 Heidelberg, \newline 
Tel.\ +49\,6221\,54\,5686, Fax.\ +49\,6221\,54\,8312 \newline%
{Email: \url{hofmann@mathi.uni-heidelberg.de}}%, \newline \url{www.mathi.uni-heidelberg.de/\~hofmann}\newline
}
 }
\begin{document}
\maketitle
%\tableofcontents

\section*{Abstract}
In the present paper, we provide a construction of the multiplicative Borcherds lift for unitary groups $\Ug(1,m)$,
which takes weakly holomorphic elliptic modular forms as input functions and lifts them to
automorphic forms having infinite product expansions and taking their zeros and poles along Heegner divisors.
In order to transfer Borcherds' theory to unitary groups, we construct a suitable embedding of $\Ug(1,m)$ into $\Og(2,2m)$.
We also derive a formula for the values taken by the Borcherds products at cusps of the symmetric domain of the unitary group. 
Further, as an application of the lifting, we obtain a modularity result
for a generating series with Heegner divisors as coefficients, along the lines of Borcherds' generalization of the Gross-Zagier-Kohnen theorem.   \\
{2010 \textit{Mathematics Subject Classification:} 
11F27, 11F55, 11G18, 14G35.}\\
{\textit{Key words and phrases:}  Borcherds product, unitary modular form, Heegner divisor, unitary modular variety}

\section{Introduction and statement of results}

In \cite{Bo95}, \cite{Bo98}, Borcherds constructs a multiplicative lifting taking weakly holomorphic modular forms to automorphic forms on orthogonal groups of signature $(2,b)$. 
These automorphic forms have infinite product expansions as so called Borcherds products and take their zeros and poles along certain arithmetic divisors, known as Heegner divisors. 
In the present paper, we provide a construction of this lifting for unitary groups of signature $(1,m)$.

A weakly holomorphic modular form for $\SL_2(\Z)$ is a function $f$ on the complex upper half-plane $\Hp = \{ \tau \in \C\,;\; \Im\tau > 0 \}$, 
which transforms as an elliptic modular form, is holomorphic on $\Hp$ and is allowed to have poles at the cusps.
Thus, $f$ has a Fourier expansion of the form
\[
f(\tau) = \sum_{\substack{n \in \Z \\ n \gg -\infty}} c(n)e(n\tau),
\]
where, as usual $e(n\tau) = \exp(2 \pi i n \tau)$. The principal part of $f$ is given by the Fourier polynomial $\sum_{n<0} c(n)e(n\tau)$.

Let $\F$ be an imaginary quadratic number field $\Q(\sqrt{D_\F})$, with $D_\F<0$ its discriminant. Denote by $\OF$ the ring of integers in $\F$. 
The inverse different ideal, denoted $\DF^{-1}$, is the $\Z$-dual of $\OF$ with respect to the trace $\tr_{\F/\Q}$.
 Let $D_\F$ be the discriminant of $\F$. We denote by $\delta$ the square-root of $D_\F$, with the principal branch of the complex square-root.
Then, as a fractional ideal, $\DF^{-1}$ is given by $\delta^{-1} \OF$.

Let $V=\VF$ be an $\F$-vector space of dimension $1+m$, equipped with a non-degenerate hermitian form $\hlfempty$ of signature $(1,m)$. Note that $\Qf{x} = \hlf{x}{x}$ is a rational-valued quadratic form on $\VF$.
Denote by $\VFC$, $\hlfempty$ the complex hermitian space $\VFC = \VF\otimes_\F\C$ with $\hlfempty$ extended to a $\C$-valued hermitian form.

We denote by $L$ an even hermitian lattice in $\VF$, i.e.\ an $\OF$-submodule with $L\otimes_{\OF}\F = \VF$ and with $\hlf{\lambda}{\lambda}\in\Z$ for all $\lambda \in L$. 
For the purposes of this introduction, assume $L$ to be unimodular. 
Further, we assume that $L$ contains a primitive isotropic vector $\ell$ and a second isotropic vector $\ell'$, with $\hlf{\ell}{\ell'} \neq 0$.
Additionally, in the present introduction we assume $\hlf{\ell}{\ell'} = -\delta^{-1}$.    
 
Denote by $\Ug(V)$ the unitary group of $\VF$, $\hlfempty$ and $\Ug(V)(\R)$ the set of its real points, thus $\Ug(V)(\R) \simeq \Ug(1,m)$. 
We denote by $\Ug(L)$ the arithmetic subgroup of $\Ug(V)$ which acts as the isometry group of $L$. 
We consider $\Gamma = \Ug(L)$ as the unitary modular group.

The hermitian symmetric domain for the action of $\Ug(V)(\R)$ is isomorphic to the positive projective cone
\[
\coneU = \left\{ [v]\in \PFR\,;\; \hlf{v}{v} > 0  \right\}.
\]
For each $[v] \in \coneU$, there is unique representative of the form $z = \ell' - \tau\ell + \sigma$, with $\sigma$ negative definite. 
Now, the symmetric domain can be realized as an unbounded domain, called the Siegel domain model:
\[
\HU = \left\{ (\tau,\sigma) \in \C\times\C^{m-1}\,;\; 2\Im\tau\abs{\delta}^{-2} > - \hlf{\sigma}{\sigma}  \right\}.
\]
The action of $\Ug(V)(\R)$ on $\HU$ induces a non-trivial factor of automorphy, $j(\gamma;\tau,\sigma)$. 

Let $k$ be an integer, $\Gamma$ a subgroup of finite index in $\Ug(L)$ and $\chi$ a multiplier system of finite order on $\Gamma$.
 A unitary automorphic form of weight $k$ on $\Gamma$ with multiplier system $\chi$ is a meromorphic function $f: \HU \rightarrow \C$ which satisfies
\[
f(\gamma(\tau,\sigma)) 
= j(\gamma;\tau, \sigma)^k \chi(\gamma) f(\tau,\sigma)\quad\text{for all}\; \gamma \in \Gamma.
\] 
An automorphic form is called a modular form if it is holomorphic on $\HU$ and regular at the cusps of $\HU$. 

Let $\lambda \in L$ be a lattice vector with negative norm, i.e.\ $\hlf{\lambda}{\lambda}<0$.
 The complement of $\lambda$ with respect to $\hlfempty$ defines a codimension one subset of the projective cone $\coneU$, denoted $\HeegU(\lambda)$.
The corresponding subset of $\HU$ is the support of a primitive divisor, a \emph{prime Heegner divisor}, which we also denote $\HeegU(\lambda)$.

For a negative integer $n$, the Heegner divisor $\HeegU(n)$
of index $n$ on $\HU$ is the $\Gamma$-invariant 
divisor defined as the locally finite sum 
\[
\HeegU(n) \vcentcolon =
 \sum_{\substack{\lambda \in L \\ \hlf{\lambda}{\lambda} = n}} 
\HeegU(\lambda).
\] 
Write $\OF$ in the form $\Z + \zeta\Z$, with $\zeta = \frac12 \delta$ if $\DF$ is even and $\zeta = \frac12(1 + \delta)$ if $\DF$ is odd. Then, for $L$ a unimodular lattice containing a primitive isotropic vector $\ell$ and a second isotropic vector $\ell'$, with  
$\hlf{\ell}{\ell'} = -\delta^{-1}$, our main result can be stated as follows:
\begin{theorem}\label{thm:main_intro}
Given a weakly holomorphic modular form $f$ for $\SL_2(\Z)$ 
of weight $1 - m$, with a
Fourier expansion of the form 
$\sum_{n\gg -\infty} c(n)\,e(n\tau)$. 
Assume that $f$ has integer
coefficients in its principal part. 
Then, there is a meromorphic function $\Xi_f(\tau, \sigma)$ on $\HU$
with the following properties 
\begin{enumerate}
\item $\Xi_f(\tau, \sigma)$ is an automorphic form of weight $c(0)/2$ for $\Ug(L)$. 
\item The zeros and poles of $\Xi_f$ lie on Heegner divisors. We have
\[
\Div \bigl( \Xi_f \bigr) =
 \frac12 \sum_{\substack{n<0 \\ c(n) \neq 0}} c(n)\,\HeegU(n), 
\]
with the Heegner divisors $\HeegU(n)$ as introduced above.
\item Near the cusp corresponding to $\ell$, 
the function $\Xi_f(\tau, \sigma)$ has an
absolutely converging infinite product expansion, for every Weyl chamber $W$, of the form 
\[
\begin{aligned}
\Xi_f(\tau, \sigma)
 & = C e\left( \delta \hlfa[\big]{z, \rho_f(W)}\right)
\prod_{\substack{\lambda \in K \\ \Wpos{\lambda}{W}>0}}\Bigl(1 -
 e\left( \delta \hlfa[\big]{z, \lambda} \right) \Bigr)^{c(\hlf{\lambda}{\lambda}, \lambda)}, \\
\end{aligned}
\]
where, as above, $z = \ell' - \tau\ell + \sigma$, $C$ is a constant of absolute value $1$ and $\rho_f(W)$ is the Weyl vector attached to $W$; $K$ denotes a $\Z$-submodule of $L$, of rank $2m$.
 Here, $K$ can be written in the form $K = \Z\zeta\ell \oplus \Z\ell' \oplus D$, with $D$ negative definite. The positivity condition $\Wpos{\lambda}{W}>0$ depends on the Weyl chamber. 
\item The lifting is multiplicative. We have $\Xi_{f + g}(\tau, \sigma) = \Xi_f(\tau, \sigma)\cdot \Xi_g(\tau, \sigma)$.  
\end{enumerate}
\end{theorem}
The Weyl chambers occurring here are connected components of $\HU$ defined by inequalities depending on the principal part of the input function $f$, see section \ref{sec:weylch} below.  

Viewed in the projective model $\coneU$, the %($\F$-rational) 
boundary components of the symmetric domain consist of isolated points.
The boundary point associated to $\ell$ corresponds to the cusp at infinity of $\HU$, see section \ref{sec:unitary} below. 
Since the lift $\Xi_f$ is meromorphic, a priori, it may have a pole or a zero at this cusp. 
If this not the case, the following result gives the value taken at the cusp explicitly:
 \begin{theorem}\label{thm:limit_simple} 
Let $W$ be a Weyl chamber, such that the cusp at infinity is contained in the closure of $W$. 
If this cusp is neither a zero nor a pole of $\Xi_f(\tau,\sigma)$, then the limit $\lim_{\tau\rightarrow i\infty} \Xi_f(\tau,\sigma)$ is given by
\[
\lim_{\tau\rightarrow i\infty}\Xi_f(\tau, \sigma) = C e\bigl(\overline{\rho_{f}(W)}_\ell\bigr)\prod_{\substack{\lambda = \kappa\zeta\ell \in K \\ \kappa \in \Q_{>0}} }
\left(1 - e\bigl(\kappa\bar\zeta\bigr)\right)^{c(0,\lambda)},
\] 
where $\rho_{f}(W)_\ell$ denotes the $\ell$-component of the Weyl vector $\rho_f(W)$.
\end{theorem}
In \cite{Bo99} Borcherds obtained a result, \cite{Bo99}, theorem 4.5, on the modularity of Heegner divisors, which, using Serre duality for vector bundles over Riemann surfaces, can be seen as a consequence of the existence and the properties of the Borcherds lift.

We extend a result of this type to Heegner divisors on $\HU$. Let $\UXGamma = \Gamma \backslash \HU$ be the unitary modular variety and $\pi: \tilde\UXGamma \rightarrow \UXGamma$ a desingularization.
 Denote by $\Chow(\tilde\UXGamma)$ the first Chow group of $\tilde\UXGamma$. Recall that $\operatorname{Pic}(\tilde\UXGamma) \simeq \Chow(\tilde\UXGamma)$.
We consider a modified Chow group $\Chow(\tilde\UXGamma)/\BDiv$, where $\BDiv$ denotes the subgroup of boundary divisors of $\tilde\UXGamma$. 
Let $\Abdl{k}$ be the sheaf of meromorphic automorphic forms on $\UXGamma$. 
By the theory of Baily Borel, there is a positive integer $n(\Gamma)$,
 such that for every positive integer $k$ divisible by $n(\Gamma)$, the pullback of $\Abdl{k}$ defines an arithmetic line bundle in $\operatorname{Pic}(\tilde\UXGamma)$. We denote its class in the modified Chow group  by $c_1(\Abdl{K})$.
 More generally, for $k\in\Q$, we define a class in $(\Chow(\tilde\UXGamma)/ \BDiv)_\Q = (\Chow(\tilde\UXGamma)/ \BDiv) \otimes_{\Z} \Q$ by choosing an integer $n$, such that $kn$ is a positive integer divisible by $n(\Gamma)$, and setting $c_1(\Abdl{k}) = \frac{1}{n} c_1(\Abdl{kn})$.
As the Heegner divisors are $\Q$-Cartier on $\UXGamma$, their pullbacks define elements in $(\Chow(\tilde\UXGamma)/\BDiv)_\Q$. 

For a unimodular lattice $L$, our modularity result can now be formulated as follows.
\begin{theorem}\label{thm:Heeg_simple}
The generating series
\[
A(\tau) = c_1(\Abdl{-1/2})  + \sum_{n\in\Z}\pi^*(\HeegU(n))\,q^n\; \in\Q[[q]]\otimes(\Chow(\tilde\UXGamma)/\BDiv)_\Q  
\]
is a modular form valued in $(\Chow(\tilde\UXGamma)/\BDiv)_\Q$. More precisely, $A(\tau)$ is an element of $\mathcal{M}_{1+m}(\SL_2(\Z))\otimes (\Chow(\tilde\UXGamma)/\BDiv_\Q)$.  
 \end{theorem}
We derive a more general version of this theorem, where $L$ is not assumed to be unimodular and $\Gamma$ is an arbitrary unitary modular group in section \ref{sec:heeg2}. Recently, a similar result has been independently obtained by Liu, in \cite{Liu}, using a completely different approach.

The paper is structured as follows:
In sections \ref{sec:unitary} and \ref{sec:ortho} we describe the basic setup and introduce notation used in following chapters.
Section \ref{sec:unitary} covers the basic theory for unitary groups, in particular the theory of hermitian lattices, the construction of the Siegel domain model and the definition of modular forms. 
In section \ref{sec:ortho} we recall the theory of modular forms on orthogonal groups of signature $(2,2m)$ and in particular the construction of the tube domain model $\HO$ for the symmetric domain.
Most of the content of this section is well-known, more detailed accounts can be found in \cite{Br02}, chapter 3, or \cite{Bo98}, section 13.

Section \ref{sec:embed} is in a  sense the technical heart of the paper. 
The $(m+1)$-dimensional hermitian $\F$-vector space $\VF, \hlfempty$ is identified with the underlying rational ${(2m+2)}$-dimensional $\Q$-vector space, denoted $\VQ$ and equipped with the bilinear form $\blfempty \vcentcolon = \tr_{\F/\Q} \hlfempty$.
 From this, we get an embedding of $\Ug(V)$ into $\SO(V)$. 
This, in turn, induces an embedding of the corresponding symmetric domains.
 Our objective is to realize this as an embedding of the Siegel domain model $\HU$ into the tube domain model $\HO$, 
which is compatible with the complex structures of both $\HU$ and $\HO$ and with the geometric structure resulting from the choice of cusps.

The following sections cover some of the prerequisites for the Borcherds lift. In section \ref{sec:weil_vv} we recall the Weil representation $\rho_L$ of $\SL_2(\Z)$
on the group algebra $\C[L'/L]$ and the definition of weakly holomorphic vector valued modular forms transforming under $\rho_L$.
In section \ref{sec:heeg1}, we introduce the Heegner divisors on $\HU$ and show how these can also be described through the restriction to the embedded $\HU$ of Heegner divisors on $\HO$,
 which occur in Borcherds' theory for orthogonal groups. Finally, section \ref{sec:weylch} covers the definition of Weyl chambers of $\HU$ which is derived from that of Weyl chambers of $\HO$.

Our main theorem is presented in section \ref{sec:mainthm}, the proof given here is based on Borcherds' seminal paper \cite{Bo98}. 
We apply the machinery of the pullback under the embedding from section \ref{sec:embed} to the proof of Borcherds' theorem 13.3 from \cite{Bo98}.

In section \ref{sec:limit} we take the Borcherds products constructed in the preceding section and evaluate their limit on the boundary of $\HU$. 
We obtain a more general form of theorem \ref{thm:limit_simple} above.

In the last section, section \ref{sec:heeg2}, we derive the modularity statement for Heegner divisors sketched as \ref{thm:Heeg_simple}. 
We start with a slight reformulation of a result of Borcherds' on obstructions, see \cite{Bo99}, theorem 3.1. 
The modularity result then quickly follows from our main theorem.

\section{Hermitian spaces and modular forms for unitary groups}\label{sec:unitary}
\paragraph{The number field $\F$}
Let $\F = \Q(\sqrt{D_\F})$ be an imaginary quadratic field of discriminant $D_\F<0$, 
and denote by $\OF$ be the ring of integers of $\F$. 
We consider $\F$ as a subfield of $\C$.
Denote by $\delta$ the square root of $D_\F$, where the square root is understood as the principal branch of 
the complex square root. 

The \emph{inverse different} ideal $\DF^{-1}$ %
(also called \emph{complementary ideal}) is the $\Z$-dual of $\OF$
with respect to the trace $\tr_{\F/\Q}$. As a fractional ideal, $\DF^{-1} = \delta^{-1}\OF$.

We may write $\OF$ in the form $\Z + \zeta\Z$, where for the basis element $\zeta$, we set
\begin{equation}\label{eq:def_zeta}
 \zeta = \begin{dcases*}
 \delta/2 & \text{if $D_\F$ is even,} \\
 \frac12 (1 + \delta) & \text{if $D_\F$ is odd.} 
 \end{dcases*}
\end{equation}
\begin{rmk*}
Another common choice of basis element is $\frac12 ( D_\F + \delta)$. The above choice of $\zeta$, which depends on the parity of $D_\F$, is technically motivated, as it avoids the need for frequent case distinction and allows for a more concise notation, particularly in sections 4, 8 and 9 below.  
\end{rmk*}

\paragraph{A hermitian space}
Let $\VF$ be a vector-space of dimension $1+m$ over $\F$, equipped with
a non-degenerate hermitian form $\hlfempty$, indefinite of signature
$(1,m)$, with $m\geq 1$. We define $\hlfempty$ to
be linear in the first and conjugate-linear in the second argument, 
\[
\hlf{\alpha x}{\beta y} = \alpha \bar\beta \hlf{x}{y}, \qquad
\text{for all}\quad \alpha,\beta \in \F.
\] 
By $\VFC$ we denote the complex hermitian space $\VF\otimes_\F\C$, with 
$\hlfempty$ extended to a complex hermitian form.

Note that with the non-degenerate, symmetric bilinear form is defined as $\blfempty\vcentcolon = \tr_{\F/\Q}\hlfempty$, $\VF$ becomes a quadratic space over $\Q$.
The corresponding real quadratic space is denoted $\VQR$, $\blfempty$, where the extension of $\blfempty$ to a bilinear form over $\R$ is equal to $2\Re\hlfempty$. 

\paragraph{Hermitian lattices} 
By a \emph{hermitian lattice} in $\VF$, we mean an $\OF$-submodule $L$ of $\VF$ with
$L \otimes_{\OF} \F = \VF$, equipped with the hermitian form $\hlfempty\mid_L$.

A lattice $L$ is called integral, if $\hlf{\lambda}{\mu} \in \DF^{-1}$ for any $\lambda, \mu \in L$.
Further, $L$ is called even, if $\hlf{\lambda}{\lambda} \in \Z$, for any $\lambda \in L$.

The dual lattice $L'$ is defined as 
\[
L' = \{ v \in \VF\,;\, \hlf{\lambda}{v} \in \DF^{-1} \quad\text{for all}\quad \lambda \in L\}.
\]  
If $L$ is an integral lattice, then $L\subset L'$. A lattice is called unimodular, if it is self-dual.

Let $L$ be an integral hermitian lattice. The quotient $L'/L$ is a finite
$\OF$-module, called the \emph{discriminant group} of $L$.

\begin{rmk*} If $\VF$ is considered as a rational quadratic space with the bilinear form $\blfempty$, 
a lattice $L$ is a quadratic $\Z$-submodule of maximal rank. 
If a hermitian lattice $L$ is integral or even by the above definition, it is also integral by the usual definition for quadratic lattices as $\Z$-modules. Also, the dual of $L$ with respect to $\blfempty$ is also given by $L'$.
\end{rmk*}

Now, assume that $\VF$ contains an isotropic vector. (This assumption is trivial if $m>2$.) Then,
 if $L$ is an (even) hermitian lattice in $V$, we may choose a primitive isotropic vector $\ell$ in $L$ and
there exists a lattice vector  $\ell' \in L'$ with $\hlf{\ell}{\ell'} \neq 0$. 
\begin{assume}
 In the following, whenever dealing with a hermitian lattice $L$, we will assume the following
\begin{enumerate}
\item $L$ is even.
 \item There is a primitive isotropic vector $\ell \in L$.
\item The dual lattice contains a vector $\ell'$ with $\hlf{\ell}{\ell'}\neq 0$ and {which is also isotropic}. 
\end{enumerate}
\end{assume}

\begin{example}\label{ex:hyp_plane}
 Consider the $\OF$-module $H = \OF\oplus\DF^{-1}$, generated by
$1$ and $\delta^{-1}$. This is a lattice in $\VF = \F^2$. 
With the indefinite hermitian form $\hlfempty$ given by
\[
\hlf{ (x,y)}{ (x',y')} = x\bar{y}' + y\bar{x}', 
\]
the lattice $H$ has signature $(1,1)$, is even and unimodular. A
hermitian lattice isomorphic to $H$ is called a
\emph{hyperbolic plane over $\OF$}. 
Here, we may choose $\ell = (1,0) \in \OF$ and $\ell'$ any element of $\DF^{-1}$.
\end{example}

Finally, for the hermitian lattice $L$, with $\ell$ and $\ell'$ as above,
 we denote by $D$ the complement of $\ell$ and $\ell'$ with respect to $\hlfempty$ in $L$,
\[
D = \left\{ \kappa \in L\,;\, \hlf{\kappa}{\ell} =\hlf{\kappa}{\ell'} = 0\right\}.
\]

\paragraph{Unitary groups}
The \emph{unitary group} of $\VF$ is the subgroup of
$\group{GL}(V)$ preserving the hermitian form,
\[
 \Ug(V) = \left\lbrace g \in \group{GL}(V)\,;\,
 \hlf{g x}{g y} = \hlf{x}{y}\quad\!\text{for all}\quad x,y\in \VF \right\rbrace,
\]
We consider $\Ug(V)$ as an algebraic group defined over $\Q$. The set of its real
points, $\Ug(V)(\R)$, is the unitary group of $\VFC$.

Let $L$ be a hermitian lattice in $\VF$. The isometries of $L$ define an
arithmetic subgroup, consisting of the integer points of $\Ug(V)$, 
\[ 
\Ug(L) = \left\lbrace g\in \Ug(V);\; g(L) = L \right\rbrace.
\]
We shall consider the subgroups of finite index in $\Ug(L)$ as \emph{unitary 
  modular groups.} 
In particular, the \emph{discriminant kernel} is the subgroup of $\Ug(L)$ acting trivially
on the discriminant group, denoted by $\GammaU{L}$.

\paragraph{The  Siegel domain model}

A standard symmetric domain for the operation of \linebreak$\Ug(V)(\R)$ is given by the quotient
\[
\Ug(V)(\R) / \mathcal{C}_{\Ug} \simeq \Ug(1,m) / \left( \Ug(1) \times \Ug(m)\right).
\]
 with $\mathcal{C}_{\Ug}$ a maximal compact subgroup. It can also be described as
the following positive cone in the projective space $\PFR$,
\[
\coneU = \left\{ [v] \in \PFR\,;\, \hlf{v}{v}  >0,\;\text{for all}\; v \in [v] \right\}.
\]
We can also realize $\coneU$ as an unbounded domain, denoted by $\HU$. To this aim, for each $[v]\in\coneU$,  we choose a representative $z$ of the following form,
\begin{equation} \label{eq:defz}
 z = \ell' - \delta\hlf{\ell'}{\ell} \tau\ell + \sigma,
\end{equation}
 with $\sigma$ in the complement $D\otimes_{\OF}\C$ of $\ell$ and $\ell'$.
 Then by the positivity condition, 
\[
\hlf{z}{z}  = 2\abs{\delta}\abs{\hlf{\ell'}{\ell}}^2 \Im\tau + \hlf{\sigma}{\sigma} > 0.
\]
The set of $(\tau, \sigma)$ satisfying this equation, is thus an affine model for $\coneU$,
\[
\HU = \left\{  (\tau, \sigma) \in  \C\times \bigl(D\otimes_{\OF}\C\bigr) \,;\, 2\abs{\hlf{\ell'}{\ell}}^2 \Im\tau> - \hlf{\sigma}{\sigma} \right\}.
\] 
We call $\HU$ the \emph{Siegel domain model}. To each pair $(\tau, \sigma) \in \HU$, we can uniquely associate a
$z\in \VFC$ with $[z]\in\coneU$ of the above form, which we denote as $z(\tau,\sigma)$.
Note that the operation of $\Ug(V)(\R)$ on $\HU$ gives rise to a non-trivial automorphy factor.

\paragraph{Parabolic subgroups and the stabilizer of a cusp} 

The choice of the isotropic vector $\ell$ corresponds to choosing a boundary point $[\ell]\in \PFR$ of $\coneU$. This is 
the cusp of $\HU$ at $[\ell]$ (called infinity). 

Denote by $P(\ell)$ the stabilizer in $\Ug(V)$ of the one-dimensional isotropic subspace  $\F\ell$ in $\VF$. We consider the following transformations contained in $P(\ell)$:
\begin{align}
\label{def:translationsU}\quad & [h,0]: v \; \mapsto \; v - \hlf{v}{\ell}\delta h \ell & \quad\text{for $h\in\Q$},    \\
\label{def:eichlerU}\quad & [0,t]:\; v \; \mapsto \; v + \hlf{v}{\ell}t - \hlf{v}{t}\ell 
        - \frac{1}{2}\hlf{v}{\ell}\hlf{t}{t}\ell & \quad\text{for $t\in D\otimes_{\OF}\F$}.    
\end{align}
Transformations of the form \eqref{def:eichlerU} are called Eichler transformations.
The action of these transformations on a point $(\tau, \sigma)$ of $\HU$ is given by
\[
[h,0]: (\tau,\sigma) \mapsto (\tau + h,\sigma), \quad
[0,t]: (\tau,\sigma) \mapsto \left( \tau +
\frac{\hlf{\sigma}{t}}{\delta\hlf{\ell'}{\ell}}
+\frac12 \frac{\hlf{t}{t}}{\delta}, \sigma + \hlf{\ell'}{\ell}t\right).
\]
The set of pairs $[h,t]$ is the \emph{Heisenberg group} attached to $\ell$, denoted $H(\ell)$. Its group law is given by 
\[
[h,t]\circ [h',t] = \bigl[h + h' +
\frac{\Im\hlf{t}{t'}}{2\abs{\delta}}, t + t'\bigr].
\]
The full stabilizer of the cusp attached to $\ell$ is the following semi-direct product 
\[
 P(\ell) = H(\ell) \ltimes  \Ug\bigl(D\otimes_{\OF} \F\bigr),
\]
with as direct factor the unitary group of the definite subspace $D\otimes_{\OF}\F$, considered as a subgroup of $\Ug(V)$.
If $\Gamma$ is a unitary modular group, i.e.\ of finite index in $\Ug(L)$, we denote by $\Gamma(\ell)$ the intersection $\Gamma \cap P(\ell)$.
\begin{prop}\label{prop:ParabN}
If $\Gamma$ is a subgroup of finite index in $\Ug(L)$, then there is a positive rational number $N$ and a lattice $\widetilde{D}$ of finite index in $D$, such that $[h,t] \in \Gamma(\ell)$, for all $h \in N\Z$, $t\in \widetilde{D}$, and $\frac12 \Im\hlf{t}{t'} \abs{\delta}^{-1}\! \in N\Z$, for all  $t,t'\in \widetilde{D}$.
\end{prop}
\begin{proof}
 Denote by $H$ the set of $h \in \Q$, such that $[h,0] \in \Gamma(\ell)$. Clearly, $\Z \subset H$, as from \eqref{def:translationsU} by the evenness of $L$ and the fact that $\OF \ell \subset L$, we get $[h,0] v \in L$ for all $v \in L$, if $h \in \Z$. On the other hand,  $H$ is an additive subgroup of $\Q$ isomorphic the center of $H(\ell)$.
 It follows that there is a unique $N \in \Q_{>0}$ with $H = N\Z$. Now, if $\Gamma$ is of finite index in $\Ug(L)$, there is a positive integer $r$ such that $[h,0]^r = [rh,0] \in \Gamma(\ell)$.     
Replace $N$ by $r\cdot N$. 

Next, consider $[0,t]$ with $t\in D$. Note that $[0,t]^k = [0,kt]$ for any $k\in\Z$.
For $v \in L$, we have $[0,t]v - v \in L'$ or $[0,2t]v - v \in L'$. Further, $L'/L$ is finite, so there is an integer $k>0$ such that $[0, kt] v \in L$ for any $v \in L$. For $\Gamma \subset \Ug(L)$ of finite index, there is a suitable multiple $k'$ with $[0,k't] \in \Gamma(\ell)$. Iterating over $t$ from a basis of $D$, the existence of a lattice $\tilde{D}$ of finite index in $D$ follows, with $[0, \tilde{D}] \subset \Gamma(\ell)$. 

Finally, the last part of the statement follows by the group law of the Heisenberg group, 
since $\bigl[\frac12 \abs{\delta}^{-1} \Im\hlf{t}{t'}, 0 \bigr] \in \Gamma(\ell).$  
\end{proof}

\paragraph{Compactification} 
Let $\Gamma$ be a unitary modular group. 

The modular variety $\UXGamma$ is given by the quotient $\Gamma \backslash \HU \simeq \Gamma\backslash\coneU \in \PFR$. 
It can be compactified as follows. The cusps of $\UXGamma$ are defined by $\Gamma$-orbits of isotropic one-dimensional subspaces of $\VF$. Denote the set of these subspaces as $\operatorname{Iso}(V)$. For each $I \in \operatorname{Iso}(V)$, a boundary point of $\coneU$ is given by $I_\C = I \otimes_\C \F$.
 Now, a compactification of $\UXGamma$, the Baily-Borel compactification, $\UXGammaBB$, can be obtained by introducing a suitable topology and a complex structure on the quotient
\begin{equation}
\Gamma \backslash \left( \coneU \cup \left\{ I_\C; I\in \operatorname{Iso}(V) \right\}\right). \label{eq:BBquotient}
\end{equation}
By the results of Baily-Borel, from \cite{BailyBorel}, it carries the structure of a normal complex space. However, in general, there are remaining singularities at the cusps. These can be resolved, see \cite{Ho11}, chapter 1.1.5, in what amounts to the construction of a toroidal compactification, cf.\ \cite{BHY13}, section 4.1. 

We briefly sketch how the Baily-Borel compactification is constructed.
Modulo $\Ug(V)$-translates, it suffices to define neighborhoods for the cusp at infinity, with $I_\C = [\ell]$:
\begin{definition}\label{def:setHUC} 
For a real number $C>0$, define the following  $\Gamma(\ell)$-invariant sets 
\[
\begin{aligned}
\HU^C & = 
\left\lbrace (\tau, \sigma) \in \HU;\;  
2\abs{\delta} \abs{\hlf{\ell'}{\ell}}^2\,\Im(\tau) + \hlf{\sigma}{\sigma} > C 
\right\rbrace, \\
& \simeq \left\lbrace z\in\coneU;\; \frac{\hlf{z}{z}}{\abs{\hlf{z}{\ell}}^2} \abs{\hlf{\ell'}{\ell}}^2  > C \right\rbrace. 
 \end{aligned}
\]
\end{definition}
Further, denote $\overline{\HU^C} = \HU^C \cup \{ \infty \}$. 
A topology on $\HU \cup \{\infty\}$ is defined as follows: 
 A subset $U \subset \HU \cup \{\infty\}$ is called \emph{open}, if, 
 firstly, $U \cap \HU$ is open in the usual sense and, secondly, 
if $\infty \in U$ implies that $\overline{\HU^C}  \subseteq U$ for some $C>0$.

By similarly introducing neighborhoods in $\HU^* = \HU \cup \left\{ I_\C\, ; I \in \operatorname{Iso}(V) \right\}$ for every  $\Gamma$-equivalence class of $\operatorname{Iso}(V)$, one gets a topology on $\HU^*$. The quotient topology then defines the  usual Baily-Borel topology on  $\Gamma \backslash \HU^*$, the quotient from \eqref{eq:BBquotient}. 

The complex structure on $\UXGammaBB$ is defined as follows. 
Denote by $\operatorname{pr}$ the canonical projection of $\operatorname{pr}: \HU^* \rightarrow \Gamma \backslash \HU^*$. 
For an open set $U\subset \Gamma\backslash \HU^*$, let $U' \subset \HU^*$ be the inverse image under $\operatorname{pr}$ and let $U'$ be the inverse image of $U''$ in $\HU$. We have the following diagram
\[
\xymatrix{%
  {\HU} \ar[r] & {\HU^*} \ar[r]^{\operatorname{pr}} &
{\UXGammaBB} \\ 
 {U''} \ar@{_(->}[u] \ar[r] &  {\;U'}  \ar@{_(->}[u] \ar[r] &
**[r]{U^{\phantom{*}}.} \ar@{_(->}[u]} 
\] 
Now, define $\mathcal{O}(U)$ as the ring of continuous functions $f: U\rightarrow \C$, which have holomorphic pullback $\operatorname{pr}^*(f)$ to $U'$ (and $U''$). With the usual methods of algebraic geometry this defines the sheaf $\mathcal{O}$ of holomorphic functions on $\UXGammaBB$. Then, by theory of Baily-Borel $(\UXGammaBB, \mathcal{O})$ is a normal complex space.

\paragraph{Modular forms}
Denote by $j(\gamma; \tau, \sigma)$ the automorphy factor induced by the action of $\Ug(V)(\R)$ on $\HU$.
\begin{definition}
Let $\Gamma$ be unitary modular group, i.e.\ of finite index in $\Ug(L)$. A holomorphic automorphic form of weight $k$ and with character $\chi$ for $\Gamma$ 
is a function $f:\, \HU\rightarrow \C$,  with the following properties:
\begin{enumerate}
\item $f$ is holomorphic on $\HU$,
\item $f(\gamma(\tau,\sigma)) = j(\gamma; \tau,\sigma)^k \chi(\gamma) f(\tau,\sigma)$ for all $\gamma \in \Gamma$.
\end{enumerate}
If further, $f$ is regular at every cusp, it is called a $\emph{modular form}$.
\end{definition}
Meromorphic automorphic forms are defined similarly.
By the Koecher principle, if $m>1$, any holomorphic automorphic form is a modular form. 

Since by proposition \ref{prop:ParabN} translations of the form $\tau \rightarrow \tau + \nu$, 
with $\nu \in N\Z$ are induced from elements of the parabolic subgroup $P(\ell) \cap \Gamma$,
an automorphic form admits a Fourier-Jacobi expansion
\begin{equation}
 \label{eq:fj_exp}
 f(\tau, \sigma) = \sum_{n \in \Z + r} a_n(\sigma) e\left( 
\frac{n}{N}\tau\right), 
\end{equation}
with  $N$ as in proposition \ref{prop:ParabN} and $r$ a constant, $r\in\Q$, $r\geq 0$. 
Regularity at the cusp attached to $\ell$ means that the coefficients $a_n(\sigma)$ vanish for $n<0$. 
The coefficient $a_0(\sigma)$ is constant, thus denoted $a_0$. If $r$ is not in $\Z$, $a_0 = 0$.  
For $n\neq 0$, the coefficient functions $a_n(\sigma)$ transform as theta-functions under the Eichler transformations in $\Gamma$.  

The value of $f$ at infinity is given by 
\[
\lim_{\tau \rightarrow i\infty} f(\tau, \sigma) = a_0.
\]
It depends only on the $\Gamma$-equivalence class of $\ell$ and, in particular, does not depend on the choice of $\ell'$, cf.\ \cite{FrCub}, lemma 6.1. 

\section{Quadratic spaces and modular forms for orthogonal groups}\label{sec:ortho}
\paragraph{$\VF$ as a rational quadratic space} 
In the following, we consider $\VF$ as a quadratic space $\VQ$ 
over $\Q$ of signature $(2,2m)$,
with the non-degenerate symmetric $\Q$-valued bilinear form $\blfempty = \tr_{\F/\Q}\hlfempty$ 
and the quadratic form $\Qf{\cdot}$, given by $\Qf{x} = \frac12 \blf{x}{x}$. 
Note that $\Qf{x} = \hlf{x}{x}$.
Then $\VQR$ is the corresponding real quadratic space with the bilinear and quadratic forms obtained from $\blfempty$ and
 $\Qf{\cdot}$ by extension of scalars. 

By a quadratic lattice $M$, we mean a $\Z$-submodule of $\VQ$, with $M \otimes_\Z \Q = \VQ$, endowed with a quadratic form.
 In particular, if $L$ is a hermitian lattice as in section \ref{sec:unitary}, as a $\Z$-module,
 $L$ is a quadratic lattice in $\VQ$, $\blfempty$ with quadratic form $\Qf{\cdot}\mid_L$, and $D$ is a sub-lattice of $L$ 
both as a hermitian lattice over $\OF$ and as a quadratic lattice over $\Z$. 

The orthogonal group of $\VQ$, $\blfempty$ is given by
\[
\begin{gathered}
  \Og(V) = \{ \gamma \in \mathrm{GL}(V)\,;\, \blf{\gamma v}{\gamma w} = \blf{v}{w} \;\text{for all}\quad v,w \in \VQ \}, \\
\SO(V) = \Og(V) \cap \mathrm{SL}(V).
\end{gathered}
\]
The groups of real points of these groups are denoted $\Og(V)(\R)$ and $\SO(V)(\R)$, respectively. 
By $\SO^+(V)$, we denote the \emph{spinor kernel}, 
the kernel of the map $\theta$ in the following exact sequence,
\[
\xymatrix{1 \ar[r] &  \{ \pm 1 \} \ar[r] & 
\operatorname{Spin}(V)(\R) \ar[r] & \SO(V)(\R)
\ar[r]^\theta & \R^{\times}/(\R^\times)^2.
}
\]
The spinor kernel is the connected component of the identity in $\SO(V)(\R)$,
and consists of the orientation preserving transformations.

\paragraph{The tube domain model}
For the following, cf.\ \cite{Bo98}, mainly section 13, and \cite{Br02}, chapter 3.2.
A symmetric domain for the action of $\SO(V)(\R)$ on $\VQR$ is given by
\[
\SO(V)(\R) / \mathcal{C}_{\SO} \simeq \SO(2,2m) / \mathrm{S}\left(  \Og(2) \times \Og(2m) \right).
\]
As a real analytic manifold, this is isomorphic to the Grassmannian $\GrO$ of $2$-\-dimensional
positive definite subspaces of $\VQR$. In order to introduce a complex structure on $\GrO$, a continuously 
varying choice of oriented basis is carried out for each $v \in \GrO$:
 Write $v$ as 
\begin{equation}\label{eq:XLYLdef}
v = \R X_L + \R Y_L, \quad\text{with}\quad X_L\perp Y_L,\; X_L^2 = Y_L^2 > 0, 
\end{equation}
then $X_L$ and $Y_L$ can be considered as the real and the imaginary part of a complex vector $Z_L$, as follows.
Consider the complex quadratic space $\VQC = \VQ \otimes_\Q \C$ with $\blfempty$ extended to a $\C$-valued bilinear form of signature $(2,2m)$.
Then, 
\[
Z_L = X_L + iY_L, \quad\text{satisfies}\quad
\blf{Z_L}{Z_L} = 0, \; \blf{Z_L}{\overline{Z}_L} >0.
\] 
The set of such $Z_L$ defines a subset of the projective space $\PQC$, the positive cone $\mathcal{K}^\pm$ in the zero quadric $\mathcal{N}$, given by   
\[
\mathcal{K}^\pm = \{[Z_L]\,;\, \blf{Z_L}{\overline{Z}_L}> 0 \} \;\subset\;
\mathcal{N} = \{[Z_L]\,;\, \blf{Z_L}{Z_L} = 0 \} \subset \PQC.
\] 
The positive cone $\mathcal{K}^\pm$ has two connected components, which are interchanged by the action of $\SO(V)(\R)/\SO^+(V)(\R)$ and stabilized by  $\SO^+(V)(\R)$. 
We choose one fixed component and denote it by $\coneO$. The Grassmannian model $\GrO$ can now be identified with $\coneO$.
Clearly, for $[Z_L] \in \coneO$ if we write $Z_L = X_L +i Y_L$, 
then $v = \R X_L + \R Y_L$ is an element of $\GrO$, whereas, conversely, for $v$ in $\GrO$, we obtain a unique 
$[Z_L] = [X_L + iY_L]$ in $\coneO$. 
Hence, $\coneO$ is a projective model for the symmetric domain. 

To obtain an affine model, $\coneO$ can now be realized as a tube domain. 
 Let $e_1 \in L$ be a primitive isotropic lattice vector, i.e. $\Qf{e_1} = 0$, and choose $e_2\in L'$ 
with $\blf{e_1}{e_2}=1$. Denote by $K$ the Lorentzian $\Z$-sub-lattice $K = L \cap e_1^\perp \cap e_2^\perp$. 

Then, by requiring $\blf{X_L}{e_1} = 1$, $\blf{Y_L}{e_1} = 0$, we obtain a unique representative $Z_L$ for $[Z_L] \in \coneO$,
which we may write in the form
\[
Z_L = e_2 + b e_1 +  Z = (1,b,Z), 
\]
with $Z \in K\otimes_\Z \C$ and $b \in \C$. Denote by $\mathcal{H}^\pm$ the set of vectors in $K\otimes_\Z \C$ with positive imaginary part,
\[
\mathcal{H}^\pm = \{ Z = X + iY \in K\otimes_\Z\C\,;\, Y^2>0 \}.
\] 
It has two connected components and is mapped biholomorphically to $\mathcal{K}^\pm$, by the map
\begin{equation}\label{eq:mapHOconeO}
Z \longmapsto [Z_L] = [(1,-\Qf{Z} -\Qf{e_2},Z)].
\end{equation}
The tube domain model $\HO$ is the connected component of $\mathcal{H}^\pm$ mapped to $\coneO \subset \mathcal{K}^\pm$.

Let $Z = X + iY$ be a point of $\HO$, and $[Z_L]$, $v = \R X_L + \R Y_L$ the corresponding 
elements in $\coneO$ and $\GrO$, respectively. We have (cf.\ \cite{Br02}, p.\ 79):
\[
\begin{aligned}
 Z_L & = \left( 1, -\Qf{Z} -\Qf{e_2},  Z\right),  \\
 X_L & = \left( 1, \Qf{Y} - \Qf{X} - \Qf{e_2}, X \right), \\
 Y_L & = \left( 0, -\blf{X}{Y}, Y\right).
\end{aligned}
\]
Note that $X_L^2  =  Y_L^2 = Y^2$, since for any $x \in \VQR$ with $\blf{x}{e} = 0$,
the projection to $K\otimes_\Z \R$, given by 
\begin{equation}\label{eq:def_pK} 
p_K\,:\; x \mapsto x - \blf{x}{e_1} e_2, 
\end{equation}
is an isometry.

\paragraph{Modular forms}
For the lattice $L$, as a quadratic $\Z$-module, 
we denote by $\SO^+(L)$ the subgroup of $\SO^+(V)$, which consists of isometries of $L$.
By $\GammaO{L}$, we denote the discriminant kernel, the subgroup of $\SO^+(L)$ acting trivially on the discriminant group $L'/L$.  
We will consider subgroups of finite index in $\GammaO{L}$ as modular groups.

The operation of $\SO(V)(\R)$ on $\HO$ induces a non-trivial factor of automorphy, denoted $J(\gamma; Z)$. 
Then, if $\Gamma \subset \GammaO{L}$ is a modular group, automorphic forms on $\HO$ for $\Gamma$ are defined similarly as automorphic forms on $\HU$ in the unitary setting, see e.g.\ \cite{Br02}, chapter 3.3.

\paragraph{Boundary components}
The Isotropic subspaces of $\VQR$ correspond to boundary components of $\coneO$ in the zero quadric $\mathcal{N}$.
These can be described as follows cf.\ \cite{BrFr}, section 2.
\begin{rmk}
 \begin{enumerate}
\item Let $F \subset \VQR$ be a one-dimensional isotropic (real) subspace. Then, $F$ represents a boundary point of $\coneO$. A boundary point of this type is called \emph{special}. A zero-dimensional boundary component is a set
consisting of one special boundary point. 
Boundary points, which are not special, are called \emph{generic}.
\item Let $F \subset \VQR$ be a two-dimensional totally isotropic (real) subspace. 
Now, the set of all \emph{generic} boundary points, which can be represented by elements of $F\otimes_\R\C$ is called a one-dimensional boundary component. 
Note that a one-dimensional boundary component can be identified with a copy
of the usual complex upper half-plane $\Hp = \{ \tau \in \C\,;\, \Im\tau>0\}$.
\end{enumerate}
There is a one-to-one correspondence between boundary components and isotropic subspaces of $\VQR$ of the corresponding dimensions. % 
The boundary of $\coneO$ is the disjoint union of all zero- and one-dimensional boundary components.
\end{rmk}
A boundary component is called \emph{rational}, if the corresponding isotropic subspace $F$ is defined over $\Q$.
Cusps, as usual, are defined by $\Gamma$-equivalence classes of rational isotropic vectors. In particular, 
$e_1$ can be identified with a cusp for $\GammaO{L}$ of $\HO$.

\section{The embedding}\label{sec:embed}

\paragraph{The underlying quadratic space}

As already sketched above, the hermitian space $\VF, \hlfempty$ over $\F$ can also be considered as a quadratic space $\VQ$ over $\Q$, 
equipped with the bilinear form $\blfempty = \tr_{\F/\Q}\hlfempty$.
 Accordingly, we can view the complex hermitian space $\VFC, \hlfempty$ as the real quadratic space $\VQR$, with $\blfempty$ extended to a real-valued bilinear form.

The identification of $\VF, \hlfempty$ with the underlying rational quadratic space $\VQ, \blfempty$ induces am embedding of the 
unitary group $\Ug(V)$ into the special orthogonal group $\SO(V)$ associated with the bilinear form $\blfempty$. 
Further, we consider $\VQR, \blfempty$ as a real quadratic space underlying the complex hermitian space $\VFC, \hlfempty$, and we may, in particular, 
identify $\Ug(V)(\R)$  with a subgroup of $\SO(V)(\R)$.  

Similarly, a hermitian lattice $L$ in $\VF, \hlfempty$ is also a lattice in $\VQ, \blfempty$ as a quadratic $\Z$-module with
 $L \otimes_\Z \Q = \VQ$ and the quadratic form $\QfNop(\cdot)\mid_L$.
The dual of $L$ as a hermitian lattice over $\OF$ and as a quadratic lattice over $\Z$ is the same, namely $L'$.
 Thus, $L'/L$ is the discriminant group, either way.
The discriminant kernel $\GammaU{L} \subset \Ug(L)$ is a subgroup of the discriminant kernel $\GammaO{L} \subset \SO(L)$.  

On the complex hermitian space $\VFC$, of course, complex numbers act as scalars. If, however we consider the underlying real vector space, 
 an element of $\C\setminus\R$ induces a non-trivial endomorphism.  For this reason, we introduce the following notation
\begin{definition}\label{def:notationhat}
Let $\mu$ be in $\C \setminus \R$. We denote by $\hat\mu$ the endomorphism of $\VQR, \blfempty$ induced from the 
scalar multiplication with $\mu$ in the complex hermitian space $\VFC$, $\hlfempty$. 
For typographic reasons, the endomorphism induced by the complex unit $i$ is denoted $\hat\imath$. 
Note that $\hat\imath \in \SO(V)(\R)$. 
\end{definition}
Clearly, if $\mu \in \F$, then $\hat\mu$ is already an endomorphism of $\VQ$, $\blfempty$ defined over $\Q$.

\paragraph{Embedding of symmetric domains}
If $\mathcal{C}_{\Ug}$ is a maximal compact subgroup of $\Ug(V)(\R)$, embedded into $\SO(V)(\R)$, there exists a maximal compact subgroup $\mathcal{C}_{\SO}$ of $\SO(V)(\R)$ with $\mathcal{C}_{\Ug} \hookrightarrow \mathcal{C}_{\SO}$. 
 Then, we may embed the respective symmetric domains given by the group quotients
\begin{equation}
\label{eq:Grp_quotients}
 \Ug(V)(\R) / \mathcal{C}_{\Ug} \hookrightarrow \SO(V)(\R) / \mathcal{C}_{\SO}.
\end{equation}
The quotient on the left hand side of \eqref {eq:Grp_quotients} is isomorphic to a Grassmannian $\GrU$ consisting of positive definite lines 
$\C z \subset \VFC$, with $\hlf{z}{z}>0$. Thus, \eqref{eq:Grp_quotients} induces an embedding of  $\GrU$ into $\GrO$. 
Further, since, clearly,  $\GrU \simeq \coneU$, we may embed  $\coneU$ into $\GrO$, identifying an element $[z]\in\coneU$ with a subspace $v = \R X_L + \R Y_L$ contained in $\GrO$. With the bijection between $\GrO$ and $\coneO$ we finally get an embedding of $\coneU$ into $\coneO$.
A priori, this embedding is merely real analytic. 
However, by carefully choosing the oriented basis vectors $X_L$, $Y_L$ for the image of $[z]$ in $\GrO$, it can realized as a holomorphic embedding.
Then, with suitable coordinates, we also obtain an embedding of $\HU$ into the tube domain $\HO$.

To summarize, the embedding in \eqref{eq:Grp_quotients} induces maps at all levels of the following diagram: 
 \begin{equation}\label{eq:diagembed}
\xymatrix{
**[l] \C z\in \GrU \ar@{^(->}[rr] & & **[r] \GrO \ni v  \\
**[l] [z]\in \coneU \ar@{<->}[u] \ar@{^(->}[rr]^{\alpha} \ar@{.>}[rru] & &
**[r] \coneO \ar@{<->}[u] \ni [Z_L]\\
**[l](\tau,\sigma)\in \HU \ar@{<->}[u] \ar@{^(->}[rr]& & **[r] \HO
\ar@{<->}[u] \ni Z.
}
\end{equation}
To facilitate notation, we denote all these maps by $\alpha$. Which of them is the actual map under 
consideration should always be clear from the context.

\paragraph{Choice of cusp}
We want to explicitly describe the embedding of the Siegel domain model $\HU$ into the tube domain model $\HO$.
Assume that we are given $\HU$ with the choice of an isotropic vector $\ell \in L$ and of $\ell' \in L'$ with 
$\hlf{\ell}{\ell'} \neq 0$ made once and for all. 

For the construction of $\HO$ a choice of isotropic vector $e_1$ from $L$ is required, where,
 now, $L$ is viewed as a quadratic lattice over $\Z$.
We set $e_1 \vcentcolon = \ell$. 
With this choice, we associate to $\ell$ the cusp at infinity for both $\HU$ and $\HO$.
\begin{rmk*}
With this definition, the parabolic subgroup $P(\ell) \subset \Ug(V)$ stabilizing $\ell$ is mapped into 
the stabilizer of $e_1$ in $\SO(V)$. In particular, the elements of the Heisenberg group $H(\ell)$ 
are mapped to transformations generated by Eichler elements in $\SO^+(V)$. For example, it is easily verified that a translation of the forms
$[h,0]$ can be identified with an Eichler element of the form $E(\ell, \frac{h}{2}\hat\imath\ell)$, with the notation from \cite{Br02}, p.\ 86.
\end{rmk*}  

For each $[z] \in \coneU$ we have a unique representative $z \in \VFC$, with $z = z(\tau,\sigma)$ of the form \eqref{eq:defz}.
 Then, the line $\C z$, considered as a subspace of the real quadratic space $\VQR, \blfempty$, is a subspace $v \in \GrO$.
 For each $v$ the choice of an oriented basis $(X_L, Y_L)$, of the form \eqref{eq:XLYLdef}, determines an element of $[X_L + iY_L = Z_L ]$ in $\coneO$.
Next, we want to fix this choice. 

\paragraph{Complex structure}
 Consider the following diagram: 
\begin{equation}\label{eq:diagcmplx}
\begin{xy}
\xymatrix{
z \ar[r]  \ar[d]_i &  (X_L, Y_L) \ar[d]  \ar[r] &  Z_L  \ar[d]^i & **[l] = \phantom{-}X_L + iY_L\\
iz \ar[r] & (\hat\imath X_L, \hat\imath Y_L) \ar[r]  & i Z_L & **[l] = -Y_L + iX_L
}
\end{xy}
\end{equation}
To the left and to the right of the diagram, the complex unit $i$ acts as a scalar of the complex spaces $\VFC, \hlfempty$ and  $\VQC$, respectively.
By definition it acts as the endomorphism $\hat\imath$ on the vectors $X_L$, $Y_L$ in the real space $\VQR, \blfempty$. 
Note that all arrows in \eqref{eq:diagcmplx} represent $\R$-linear maps.
Thus, if \eqref{eq:diagcmplx} commutes, the following diagram also commutes 
for every $\mu \in \C\setminus\{0\}$
\[ 
\begin{xy}
\xymatrix{ 
[z] \ar@{=}[d]  \ar[r]^\alpha &  [Z_L] \ar@{=}[d]\\
[\mu z] \ar[r]^\alpha & [\mu Z_L]
}
 \end{xy}
\]
and the embedding $\alpha: \coneU \hookrightarrow \coneO$, $[z] \mapsto [Z_L]$ is 
a homomorphism between the complex projective spaces $\PFR$ and $\PQC$. 
Moreover, the induced embedding between the affine models $\HU$ and $\HO$ in \eqref{eq:diagembed} is then holomorphic. 

Since $X_L$ and $Y_L$ are contained in $\C z$, we set $X_L = \hat\psi z$. Then, clearly \eqref{eq:diagcmplx} commutes exactly if $\hat\imath X_L = - Y_L$.
Also, in this case,  $X_L \perp Y_L$ and $X_L^2  = Y_L^2 = \abs{\psi}^2 \hlf{z}{z}>0$, as required in \eqref{eq:XLYLdef}.

\paragraph{Normalization with respect to $e_1$}
Hence, with $X_L = \hat\psi z$, $Y_L = - \hat\imath\hat\psi z$ the point $[Z_L = X_L + iY_L] \in \PQC$ lies in the positive cone $\mathcal{K}^\pm$ of the zero quadric $\mathcal{N}$. 
We may also assume that it lies in the correct connected component $\coneO$. 

 For the vector $Z_L$ to be the 
unique representative for $[Z_L]$ in the image of $\HO$ under map $\mathcal{H}^\pm \rightarrow \mathcal{K}^\pm$ from \eqref{eq:mapHOconeO}, we must 
also have
\[
\blf{X_L}{e_1} = 2\Re\hlf{\psi z}{\ell} =  1 \quad\text{and}\quad
\blf{Y_L}{e_1} = 2\Re\hlf{-i\psi z}{\ell} = 0.
\]
Thus, we set $\psi = \frac12 \hlf{\ell'}{\ell}^{-1}$.
 
To summarize, for each $[z]$, a suitable choice of basis vectors for its image $v$ in  $\GrO$, is given by
\begin{equation}\label{eq:mapztoXLYL}
X_L = \left(\frac{1}{2 \hlf{\ell'}{\ell}}\right)\sphat \,z\quad\text{and}\quad 
Y_L = \left(\frac{- i}{2\hlf{\ell'}{\ell}}\right)\sphat\, z, 
\end{equation}
where $z = z(\tau,\sigma)$ as in \eqref{eq:defz} is the usual representative for $[z]$.

\begin{lemma} \label{lemma:pbmf} 
The pullback of holomorphic (meromorphic) automorphic form on $\HO$ is a holomorphic (meromorphic) automorphic form on $\HU$.
\end{lemma}
\begin{proof}
Denote by $\widetilde\coneO$ and $\widetilde\coneU$ the preimages of $\coneO$ and $\coneU$ under the canonical projections 
$\operatorname{pr}_\F: \VFC \rightarrow \PFR$ and $\operatorname{pr}_\Q: \VQC \rightarrow \PQC$, respectively. 
Denote by $\tilde\alpha$ the map $\VFC \rightarrow \VQC$ induced from $\alpha$.
Note that with the above choice, $\tilde\alpha$ is $\C$-linear. 

Let $f:\HO \rightarrow \C$ be a holomorphic automorphic form of weight $k$ for some modular group $\Gamma \subset \GammaO{L}$. Then, $f$ can be identified with a holomorphic function $\tilde f: \widetilde\coneO \rightarrow \C$, which is $\C$-homogeneous of weight $-k$ and invariant under the operation of $\Gamma$ on $\VQC$.
 
Consider the pullback $\tilde\alpha^*\tilde f$. Since $\tilde\alpha$ is $\C$-linear, the pullback is holomorphic and also $\C$-homogeneous of weight $-k$. Further, it is invariant under $\Gamma' \vcentcolon = \Gamma \cap \Ug(L)$. Clearly $\Gamma'$ has finite index in $\GammaU{L}$, as the index of $\Gamma$ in $\GammaO{L}$ is finite and $\GammaU{L} \subset \GammaO{L}$. Then, $\alpha^*f$, the attached function on $\HU$, is a holomorphic automorphic form of weight $k$ on $\Gamma'$. 
The proof for meromorphic $f$ is similar.
\end{proof}

\paragraph{A basis for the hyperbolic part} 

In the following, we introduce coordinates for the tube domain,
in order to explicitly determine the image in $\HO$ of a point $(\tau, \sigma) \in \HU$. 

For this, we need a basis $e_1, \dotsc, e_4$ for the four dimensional real subspace of $\VQR$, $\blfempty$, defined by the $\C$-span of $\ell$ and $\ell'$.
One basis vector is $e_1 = \ell$, of course. 
We put $e_3 = -\hat\zeta\ell$, where $\zeta$ is the basis element of $\OF$ defined in \eqref{eq:def_zeta}.
Then, $F = \Q e_1 \oplus \Q e_3$ is a maximal  totally isotropic rational subspace.
There is a complementary maximal isotropic subspace $F'$,
such that $F \oplus F' = \operatorname{span}_{\F}\{ \ell, \ell'\}$, considered as a subspace of $\VQ$.

In  fact, there is a basis $e_2$, $e_4$ for $F'$, such that for $e_1, \dotsc, e_4$ the following holds:
\[
 \begin{aligned}
 \blf{e_1}{e_2} & = 1, \quad \blf{e_3}{e_4} = 1, \\ 
  \blf{e_i}{e_i} & = 0 \quad (i=1,\dotsc, 4), \\
   \blf{e_i}{e_j} & = 0 \quad  (i=1,2,\; j = 3,4).
\end{aligned}
\] 
It is easily verified that these conditions are satisfied by the following vectors:
\begin{equation}\label{eq:basis_oddevenD}
e_1 = \ell, \; e_2 = \frac{\zeta}{\delta \hlf{\ell'}{\ell}} \ell',\;
e_3 = - \zeta \ell, \; e_4 = \frac{1}{\delta\hlf{\ell'}{\ell}}\ell'.
\end{equation}
Note that if $D_\F \equiv 0\mod{2}$, then $\zeta\delta^{-1} = \frac{1}{2}$.
Clearly, since $L$ is an $\OF$-module, both $e_1$ and $e_3$ are contained in $L$.
Also $e_2, e_4 \in L'$, as indeed, $\hlf{e_2}{\ell}, \hlf{e_4}{\ell} \in \DF^{-1}$.
In this basis, we can write $Z_L$ in the form 
\[
Z_L = - \left(Z_1 \cdot Z_2 + \Qf{\mathfrak{z}}\right) e_1 + e_2 + Z_1 e_3 + Z_2 e_4 + \mathfrak{z}, 
\]
with $\mathfrak{z} = \mathfrak{x} + i\mathfrak{y} \in D\otimes_\Z\C$ negative definite.

\paragraph{Embedding into the tube domain}
Now we calculate the components of $X_L$ and $Y_L$ in terms of $\tau$ and $\sigma$.
In the following calculation all complex scalars come from the complex structure of $\VFC$.
As there is no possibility of confusion, no special notation is required for these. %TYPO ``,'' OWNCORR
In the end result, all coordinates will be real. For $X_L$, we have
\begin{align}
 X_L(z) & = \frac{1}{2\hlf{\ell'}{\ell}} z =
 -\frac{\tau\delta}{2}\ell +
 \frac{1}{2\hlf{\ell'}{\ell}} \ell'  + \mathfrak{x}(\sigma) \notag \\
& = \frac{\zeta}{\delta\hlf{\ell'}{\ell}} \ell' -
  \Re\zeta\frac{1}{\delta\hlf{\ell'}{\ell}} \ell' 
  + \Re\tau( - \zeta\ell) +
 \bigl( \Re\zeta\cdot\Re\tau  + \frac12\Im\tau\abs{\delta}\bigr)\ell  +
 \mathfrak{x}(\sigma) \notag \\
 & = e_2  - \Re\zeta e_4 + \Re\tau e_3   +
  \bigl( \Re\zeta\cdot\Re\tau  + \Im\zeta\cdot\Im\tau\bigr)e_1 +
 \mathfrak{x}(\sigma),\label{eq:XLofz}
\end{align}
since $\Im \zeta = \frac12 \abs{\delta}$  and $\Re\zeta$ is either $0$ or $\frac12$, depending on the discriminant of $\F$.

Similarly, for $Y_L$ we have
\begin{align}
 Y_L(z) & = \frac{-i}{2\hlf{\ell'}{\ell}} z =
 \frac{\abs{\delta}}{2\delta\hlf{\ell'}{\ell}}\ell'  
- \Re\tau\frac{\abs{\delta}}{2}\ell  
- \Im\tau\frac{\delta}{2}\ell + \mathfrak{y}(\sigma)\notag\\
\label{eq:YLofz} & = \Im\zeta e_4 + \Im\tau e_3 -
\bigl(\Re\zeta\cdot\Im\tau - \Im\zeta\cdot\Re\tau \bigr) e_1 +
\mathfrak{y}(\sigma).
\end{align}
Now, $Z_L = X_L + iY_L$ can be written in the form
\begin{equation}\label{eq:ZLofz}
Z_L(\tau, \sigma) = \bar\zeta\tau e_1 + e_2 + \tau e_3 - \bar\zeta e_4 +
\mathfrak{z}(\sigma).
\end{equation}
Note that all coordinates in this equation are scalars of $\VQC$.
The corresponding element $Z$ in the tube domain is given by
\begin{equation}\label{eq:Zoftausigma}
Z(\tau,\sigma) = \tau e_3 -\bar\zeta e_4 + \mathfrak{z}(\sigma) = (\tau, -\bar\zeta, \mathfrak{z}). 
\end{equation}
\begin{rmk*}
 Note that the imaginary part of the $e_3$- and $e_4$-components of $Z$ is positive, due to the choice of sign for $e_3$ in \eqref{eq:basis_oddevenD}.
 Thus $\HU$ is mapped into the connected component of $\mathcal{K}$ containing the point $[1,1,i,i,0]$.  
\end{rmk*}

Finally, from \eqref{eq:mapztoXLYL}, the image of $\sigma$ under $\alpha$, is given by  $\mathfrak{z} = \mathfrak{x} + i\mathfrak{y}$, with
\[
\begin{gathered}
\mathfrak{x} = \left(\frac{1}{2 \hlf{\ell'}{\ell}}\right)\sphat \,\sigma,  \quad
\mathfrak{y}  = \left(\frac{- i}{2\hlf{\ell'}{\ell}}\right)\sphat\, \sigma.
\end{gathered}
\]
\begin{rmk*}
By the definition of $\mathcal{K} \subset \quadN$, the image $\mathfrak{z}(\sigma)$ of $\sigma$ satisfies the following set of equations,
which further restrict the geometric locus in $\HO$ of the embedded image of $\HU$:
\[
\mathfrak{z}(\sigma)^2 = \mathfrak{x}(\sigma)^2 - 
\mathfrak{y}(\sigma)^2 = 0, \quad\text{and}\quad
\blfp[\big]{\mathfrak{z}(\sigma) , \overline{\mathfrak{z}}(\sigma)} =
\mathfrak{x}(\sigma)^2  + \mathfrak{y}(\sigma)^2 
 < 0.
\]  
\end{rmk*}

\paragraph{The embedding on the boundary}
\begin{prop}
Boundary points of $\HU$ are mapped to one-dimensional boundary components of
$\HO$.
The boundary point attached to the primitive isotropic lattice vector $\ell$ is
mapped into the boundary component attached to the rational isotropic subspace $F
= \Q\ell \oplus \Q\hat\zeta\ell$ of $\VQ, \blfempty$.
\end{prop}
\begin{proof}
The boundary points of $\HU$ can be described by isotropic lines of the form $\C x$, 
with $x\in \VFC$, $x\neq 0$ and $\hlf{x}{x}=0$. Then, in the quadratic space $\VQR, \blfempty$, the two vectors $x$ and $\hat\xi x$, 
are isotropic and, for $\xi\not \in\R$, linear independent. Hence,
$F_\R \vcentcolon = \R x \oplus \R\hat\xi x$ is a two-dimensional isotropic subspace of $\VQR, \blfempty$ 
and corresponds to a boundary component of $\HO$. 

In particular, $F\vcentcolon = \F\ell$ is a two-dimensional isotopic subspace of $\VQ, \blfempty$, 
 spanned by $\ell$ and $\hat\zeta\ell$, and defines a rational boundary component.
\end{proof}
We now consider how the neighborhood of a cusp behaves under the embedding
$\alpha$. It suffices to consider the cusp at infinity.
\begin{lemma}\label{lemma:limit_embed}
 Consider the boundary point at infinity of $\HU$, attached to $\ell$. The inverse image of every open
neighborhood of the boundary point $\alpha(\infty)$ in the closure of $\coneO$
contains an open neighborhood of infinity in $\HU \cup \{ \infty \}$. 
\end{lemma}
\begin{proof}
 Consider the two-dimensional isotropic subspace $F = \Q e_1 \oplus \Q e_2$ of $\VQ, \blfempty$ and let $F_\C = F\otimes_\Q\C$.
Let $[x]$ be a point in the one-dimensional boundary component of $\coneO$ defined by $F_\C$. 
Denote by $\pi$ the canonical projection from $\VQC$ to $\PQC$.
 In the zero quadric $\quadN$, a neighborhood of $[x]$ is a union of the form $U \cup V$, with $U$ open in $\coneO$ and $V$ a subset of
$\pi(F_\C) \cap \partial\coneO$, open with respect to the subset topology. 
A more precise description can be obtained as follows, cf.\ \cite{BrFr} section 3: 
We identify $F_\C \cap  \pi^{-1}(\partial\coneO)$ with the upper half-plane $\Hp \subset \C$ via
$\tau' \mapsto -\tau' e_1 + e_3 = x$. Then, a fundamental system of neighborhoods for $x$ is given by $U_\epsilon \cup V_\epsilon$, $\epsilon>0$, with
\begin{gather*}
V_\epsilon(x) = \left\{ Z_2 \in \Hp\,;\; \abs{Z_2 - \tau'} < \epsilon   \right\}, \label{eq:V_eps} \\
U_\epsilon(x) = \left\{  (Z_1, Z_2, \mathfrak{z}) \in \HO \,;\; Z_2 \in V_\epsilon(x),\; Y_1 Y_2 + \Qf{\mathfrak{y}} > \epsilon^{-1} \right\}. \label{eq:U_eps}
\end{gather*}
Now, let $x$ be the image under $\alpha$ of the boundary point $\infty$ of $\HU$ in $\pi(F_\C) \cap \partial\coneO$. Since
\begin{equation}\label{eq:limit_alpha}
\lim_{t\rightarrow\infty}\alpha(z(it,\sigma)) = \lim_{t\rightarrow\infty}\left[ it\bar\zeta,1,it,-\bar\zeta,\alpha(\sigma) \right] = [\bar\zeta, 0, 1, 0,0],
\end{equation}
we have $x  = \bar\zeta e_1 + e_3$. 
Now, for every $Z\in\alpha(\HU)$, clearly $Z_2 = -\bar\zeta \in V_\epsilon(x)$ for all $\epsilon>0$.   
 For $Z \in \alpha({\HU})$, the imaginary part $Y$ is given by 
\[
Y = \Im\tau e_3 + \frac{\abs{\delta}}{2} e_4 + \Bigl( \frac{- i}
{2\hlf{\ell'}{\ell}}\Bigr)\sphat \sigma,
\]
for some $(\tau, \sigma) \in \HU$. Thus, if $Z \in U_\epsilon(x)\cap \alpha({\HU})$, we have
\[ 
\frac{1}{2}\Im\tau\abs{\delta} + \frac{\hlf{\sigma}{\sigma}}{4\abs{\hlf{\ell'}{\ell}}^2 } > \frac{1}{\epsilon}.  
\]
It follows that $(\tau, \sigma)$ is contained in one of the neighborhoods of
infinity $\HU^C$, as introduced in definition \ref{def:setHUC}, 
\[
 \HU^C = \bigl\{ (\tau, \sigma) \in \HU \,;\;  
2\Im\tau\abs{\delta}\abs{\hlf{\ell'}{\ell}}^2 + \hlf{\sigma}{\sigma} > C
\bigr\}.
\]
Where, in this case, $C = 4\abs{\hlf{\ell'}{\ell}}^2 /\epsilon$.
\end{proof}

\section{Weil representation and vector valued modular forms}\label{sec:weil_vv}

In this  section we recall some facts about the Weil representation and about 
vector valued modular forms.

\paragraph{The Weil representation}
Let $L$ be an even lattice over $\Z$ with symmetric non-de\-ge\-ne\-rate bilinear form $\blfempty$ and  
associated quadratic form $\QfNop(\cdot)$. Assume that $L$ has rank $2+2m$ and signature $(2,2m)$.

Then, the Weil representation of the metaplectic group $\Mp_2(\Z)$ on the group algebra $\C[L'/L]$ 
factors through $\SL_2(\Z)$. Thus, we have a unitary representation $\rho_L$ of $\SL_2(\Z)$ on $\C[L'/L]$, 
defined by the action of the standard generators of $\SL_2(\Z)$,
 $S = \left(\begin{smallmatrix} 0 & -1 \\ 1 & 0 \end{smallmatrix}\right)$ and 
$T=\left(\begin{smallmatrix} 1 & 1 \\ 0 & 1 \end{smallmatrix}\right)$, cf.\ \cite{Bo98}, section 4: 
\begin{equation}\label{eq:weilrepST} 
\begin{split}
\rho_L(T)\,\ebase_\gamma & = e\bigl( \Qf{\gamma}
\bigr)\ebase_\gamma, \\
\rho_L(S)\,\ebase_\gamma & = \frac{\sqrt{i}^{b-2}}{\sqrt{\abs{L'/L}}} 
\sum_{\delta \in L'/L}  e\bigl( -\blf{\gamma}{\delta} \bigr) \ebase_\delta,
\end{split}
\end{equation}
where $\ebase_\gamma$, $\gamma\in L'/L$ is the standard basis for $\C[L'/L]$.  
The negative identity matrix $-E_2 \in \SL_2(\Z)$ acts as $\rho_L(E_2)\ebase_\gamma = (-1)^{m-1} \ebase_{-\gamma}$.
We denote by $\rho_L^*$ the dual representation of $\rho_L$. 

\paragraph{Vector valued modular forms}
Let $\kappa \in \Z$ and $f$ be a $\C[L'/L]$ valued function on
$\Hp$ and let $\rho$ be a representation of $\SL_2(\Z)$ on $\C[L'/L]$. 
For $M \in \SL_2(\Z)$  we define the usual Petersson slash
operator as
\begin{equation}\label{eq:defPetWeil}
\left(f \mid_{\kappa, \rho} M \right)(\tau) 
= \phi(M,\tau)^{-\kappa} \rho(M)^{-1} f(M\tau),
\end{equation}
where $M=\left(\begin{smallmatrix} a & b \\ c & d   \end{smallmatrix}\right)$ acts as usual on $\Hp$, with $M\tau = \frac{a\tau+b}{c\tau +d}$ and $\phi(M,\tau) = c\tau + d$.
\begin{definition}
Let $\kappa \in \Z$. A function $f: \Hp \rightarrow \C[L'/L]$
is called a weakly holomorphic (vector valued) modular form of weight $\kappa$ transforming under $\rho$, if it satisfies
\begin{enumerate}
\item $f\mid_{\kappa, \rho} M = f$ for all $M \in \SL_2(\Z)$, 
\item $f$ is holomorphic on $\Hp$,
\item $f$ is meromorphic at the cusp $i\infty$.
\end{enumerate}
The space of such forms is denoted  $\Mweak_{\kappa, \rho}$.
The holomorphic modular forms, i.e.\ those holomorphic at the cusp, are denoted $\mathcal{M}_{\kappa, \rho}$.
\end{definition}
If $f\in \Mweak_{\kappa,\rho}$, by invariance under the $\mid_{\kappa,\rho}$-action of $T \in \SL_2(\Z)$,
 it has a Fourier expansion of the form
\[
f(\tau) = \sum_{\gamma \in L'/L} \sum_{n\in \Q}
c(n,\gamma) e(n\tau) \ebase_\gamma. 
\]   
Note that for $\rho = \rho_L$, the coefficient $c(n,\gamma)$ vanishes unless $n\equiv \Qf{\gamma} \bmod{1}$.
For $\rho_L^*$, the coefficient $c(n,\gamma)=0$ unless $n\equiv -\Qf{\gamma}\bmod{1}$.

In the present paper, weakly holomorphic modular forms of weight $\kappa = 1 - m$ are used as inputs for the Borcherds lift. 
Note that if $f \in \Mweak_{1-m, \rho_L}$, the Fourier coefficients of $f$ satisfy $c(n,\gamma) = c(n, -\gamma)$ for
$n \in \Z$ and $\gamma \in L'/L$, $\gamma \not \equiv 0 \mod L$ . 

The space $\mathcal{M}_{1+m, \rho_L^*}$ of holomorphic modular forms transforming under the dual representation also plays an important role in this paper, 
as it is the orthogonal complement of $\Mweak_{1-m, \rho_L}$ under the residue-pairing, see section \ref{sec:heeg2}. 

\section{Heegner divisors}\label{sec:heeg1}
The zeros and poles of the Borcherds lift lie along certain divisors related to lattice vectors of negative norm, 
which we define in this section.

Let $\lambda \in L'$ be a lattice vector of negative  norm, with $\hlf{\lambda}{\lambda} = \Qf{\lambda} < 0$.
Then, we associate to $\lambda$ a \emph{prime Heegner divisor} on $\HU$.
First however, we briefly recall the definition of Heegner divisors on $\HO$.
\paragraph{Heegner divisors on $\HO$}
For this, compare \cite{Br02} chapter 3 and chapter 5, also \cite{Bo98}, section 13. The orthogonal complement $\lambda^\perp$ with respect to $\blfempty$ is 
a subspace of $\VQR$, $\blfempty$ of codimension $1$ and signature ${(2,2m-1)}$. The set of subspaces $\VQ$ 
contained in $\GrO$ with $v\in\lambda^\perp$
is a sub-Grassmannian of codimension $1$, which we also denote by $\lambda^\perp$. 
Under the map from $\GrO$ to $\coneO$, $\lambda^\perp$ bijects to the set $\{ [Z_L] \in \coneO\,;\, \blf{\lambda}{Z_L} = 0\}$. 

Thus, we can identify $\lambda^\perp$ with a subset of the tube domain. Also denoted by $\lambda^\perp$, it is defined as follows:
Write $\lambda$ in the form $a e_2 + b e_1 + \lambda_K$, with $\lambda_K \in K\otimes_\Z\Q$. Then 
\[
\lambda^\perp \vcentcolon = \{ Z \in \HO\,;\, b  -\Qf{Z}a + \blf{\lambda_K}{Z} = 0 \}.
\] 
Given $\beta \in L'/L$, $n \in \Z$, with $n<0$, a Heegner divisor of index $(n,\beta)$ is a  $\GammaO{L}$-invariant divisor on $\HO$ defined by the (locally finite) sum
\[
\HeegO (n,\beta) = \sum_{\substack{ \lambda \in \beta + L \\ \QfNop(\lambda) = n}} \lambda^\perp.
\]   
Its support is a locally finite union of sub-Grass\-mann\-ians,
\[
\bigcup_{\substack{ \lambda \in \beta + L \\ \QfNop(\lambda) = n}} \lambda^\perp.
\]

\paragraph{Heegner divisors on $\HU$}
The complement of $\lambda$ in the hermitian space $\VFC, \hlfempty$ is a subspace of codimension $1$. 
The complement of $\lambda$ in the Grassmannian $\GrU$ is a closed analytic subset of codimension $1$,  which we denote as follows
\[
\begin{aligned}
\HeegU(\lambda)& \vcentcolon = \{ v \in \GrU\,;\, \hlf{\lambda}{v} = 0 \} \subset \GrU \\
& = \{ [z] \in \coneU\,;\, \hlf{\lambda}{z} = 0\} \subset\PFR.
\end{aligned}
\]
Considering representatives for $[z] \in \coneU$ of the form $z(\tau,\sigma)$, we associate to $\HeegU(\lambda)$
 a closed analytic subset of the Siegel domain, also denoted $\HeegU(\lambda)$,
\[
\begin{aligned}
\HeegU(\lambda) & \vcentcolon = \{ (\tau, \sigma) \in \HU\,;\, \hlf{\lambda}{z(\tau,\sigma)} = 0 \}.
\end{aligned}
\] 
We call a set of this form a \emph{prime Heegner divisor} on $\HU$.

Now, given $\beta \in L'/L$ and $n \in \Z$, $n<0$, we define a \emph{Heegner divisor of index $(n, \beta)$} as the (locally finite) sum 
\[
\HeegU(n,\beta) = \sum_{\substack{\lambda \in \beta + L \\ \hlf{\lambda}{\lambda} = n}} \HeegU(\lambda).
\]
Since $L$ has multiplication by $\OF$, clearly, $\HeegU(\lambda) = \HeegU(u\lambda)$ for all $u \in \OF^\times$, while $\beta = u\beta$ only if either $\beta \equiv 0 \bmod{L}$ or $u=1$.
 As a consequence, for $\beta\not\equiv 0 \bmod{L}$,
\begin{equation}\label{eq:Hulambda}
\HeegU(n,\beta) = \HeegU(n, u\beta),\quad \text{for all}\; u \in \OF^\times.
\end{equation}
In contrast, $\lambda^\perp = (\hat u\lambda)^\perp$ only if $u = \pm 1$. Thus, $\HeegO(n,\beta)$ is equal to
 $\HeegO(n, -\beta)$ but not, in general, to $\HeegO(n, u\beta)$ if $u\neq \pm 1$.

\paragraph{Behavior under the embedding}
\begin{lemma}\label{lemma:Pb_Heeg_lambda}
 Let $\lambda$ be a lattice vector with $\lambda \in L'$, with $\hlf{\lambda}{\lambda} <0$. Then, for the image of $\HeegU(\lambda)$
under $\alpha$, we have
\[
\alpha(\HeegU(\lambda)) = \left(\alpha(\HU) \cap \lambda^\perp\right) \subset \HO.
\]
The intersection is non-empty, since $\alpha$ is injective, and both $\HeegU(\lambda) \neq \emptyset$ and $\lambda^\perp \neq \emptyset$.
\end{lemma}
\begin{rmk*}
Note that, as a consequence, for all $u \in \OF^\times$ the sub-Grassmannians $(\hat u \lambda)^\perp$ intersect on the image $\alpha(\HU)$ in $\HO$.
\end{rmk*}
\begin{proof}
We consider $\lambda^\perp$ and $\HeegU(\lambda)$ as subsets of the projective cones $\coneO$ and $\coneU$.
Let $[Z_L] \in \alpha(\HU)\cap \coneO$ and assume that $[Z_L] \in \lambda^\perp$, let $[z] \in \coneU$ 
with $\alpha([z]) = [Z_L]$. Then, if $Z_L$ and $z$ are the normalized representatives, by \eqref{eq:mapztoXLYL}, we have
\[
\begin{aligned}
0  = \blf{\lambda}{Z_L} & \Leftrightarrow \blf{\lambda}{ \frac{z}{2 \hlf{\ell'}{\ell}\sphat} 
         + i \frac{ -\hat\imath z}{2 \hlf{\ell'}{\ell}\sphat}  } = 0 \\
 & \Leftrightarrow \Re\frac{ \hlf{\lambda}{z}}{\hlf{\ell'}{\ell}} + 
        i\Im\frac{\hlf{\lambda}{z}}{\hlf{\ell'}{\ell}} \\
 &\Leftrightarrow \hlf{\lambda}{z} = 0.
\end{aligned}
\]        
It follows that $z \in \HeegU(\lambda)$.
\end{proof}
Thus, a Heegner divisor on $\HO$ defines a Heegner divisor on $\HU$. And, since a linear combination of Heegner divisors on $\HO$ can be pulled back to $\HU$, the following assertion holds:
\begin{lemma}\label{lemma:PbHeeg_mbeta}
If $\HeegO(n,\beta)$ is a Heegner divisor of index $(n,\beta)$ on $\HO$, its restriction to $\HU$ is given by a Heegner divisor of the form
$\HeegU(n,\beta)$.
\end{lemma}

\section{Weyl chambers}\label{sec:weylch}
Weyl chambers are connected subsets of the symmetric domain. 
In Borcherds' theory, their definition relates to negative norm vectors in the Lorentzian lattice $K'$. 

\paragraph{The lattice $L_0$ and the projection $p$}

Consider a lattice $L$ with a primitive isotropic vector $\ell \in L$. Then, there is a unique positive integer $\widthell$,
 such that $\tr_{\F/\Q} \hlf{L}{\ell} = N_\ell \Z$. Now, a $\Z$-submodule of the dual lattice  $L'$ is defined as follows:
\[
L_0' \vcentcolon = \left\{ \lambda \in L'\,;\, \tr_{\F/\Q}\hlf{\lambda}{\ell} \equiv 0 \mod \widthell \right\}.
\]
Note that $L_0'$, is, in general, not a hermitian lattice in the sense defined in section \ref{sec:unitary}. It does however, have as multiplier ideal an order $\mathcal{O}$ in $\F$, with $\OF \supset \mathcal{O} \supset \widthell \OF$. 

Recall the definition of the Lorentzian $\Z$-lattice $K = L \cap e_1^\perp \cap e_2^\perp$, where $e_1 = \ell$ and $e_2$ is given by \eqref{eq:basis_oddevenD}.  
We have 
$\left\lvert L'/L \right\rvert =  \widthell^2 \cdot\left\lvert K'/K \right\rvert$.

For the following, see \cite{Br02}, chapter 2.1, particularly p.\ 41.
A projection $p$ from  $L_0'$ to $K'$ with the property that $p(L) = K$ can be defined as follows.
Let $f$ be a vector in $L$ with $\tr_{\F/\Q}\hlf{\ell}{f} = \widthell$ and $\lambda \in L'$.
 Denote by $f_K = p_K(f)$ and $\lambda_K = p_K(\lambda)$ the projections to the Lorentzian space $K\otimes_\Z\Q$ from \eqref{eq:def_pK}. Then, the projection
\[
p(\lambda) = \lambda_K - \frac{\blf{\lambda}{\ell}}{\widthell} f_K
\]
takes $L$ to $K$ and induces a surjective map from $L_0'/L$ to $K'/K$.
\begin{rmk}\label{rmk:beta_from_k}
Given $\kappa \in K'$ and a class $\beta + L \in L_0'/L$, with $p(\beta) = \kappa + K$, a system of representatives for $\beta$ is given by $\beta = \kappa - \blf{\kappa}{f}\ell /\widthell + b\ell/ \widthell$, where $b$ runs modulo $\widthell$, cf.\ \cite{Br02}, p.\ 45.
\end{rmk} 

\paragraph{Weyl chambers for $\HO$ and $\HU$}
Denote by $\GrK$ the Grassmannian of positive definite one-dimensional subspaces of the Lorentzian space $K\otimes_\Z\R$.
 This Grassmannian can be realized 
as a hyperboloid model, see \cite{Br02}, chapter 3.1. 
Heegner divisors in $\GrK$, as before for $\GrO$, are defined as locally finite unions of sub-Grassmannians with \emph{real} codimension one. Thus, for a negative norm vector $\kappa \in K'$, the complement $\kappa^\perp$ defines a prime Heegner divisor of $\GrK$.
Let $\beta\in L_0'/L$ and $n \in \Z$. Then, $p(\beta)$ is an element of $K'/K$, and the Heegner divisor of index $(n, p(\beta)$ is defined as the locally finite sum 
\[
\HeegO(n, p(\beta)) = \sum_{\substack{\kappa \in p(\beta) +  K \\ \Qf{\kappa} = n}}  \kappa^\perp, \quad\text{supported on}\quad\bigcup_{\substack{\kappa \in p(\beta) +  K \\ \Qf{\kappa} = n}} \kappa^\perp.
\]
Now, $\GrK - \HeegO(n,p(\beta))$ is disjoint. Its connected components are called \emph{Weyl chambers of index $(n, p(\beta))$} in $\GrK$. More generally, any non-empty finite intersection of Weyl chambers of $\GrK$ is also called a Weyl chamber.

The Weyl chambers, which appear in Borcherds theorem 13.3 from \cite{Bo98}, are connected subsets of $\HO$. 
They correspond directly to the Weyl chambers in $\GrK$, since the set of all imaginary parts of $Z=X + iY\in\HO$ is given by one of the two components of the hyperbolic cone $\{ v \in K\otimes_\R \C, v^2 > 0 \}$ and 
can be identified with the hyperboloid model of $\GrK$, via $Y \mapsto Y/\abs{Y}$, cf.\ \cite{Br02}, p.\ 78ff, for details.

Consequently, we define Weyl chambers of $\HU$ through the embedding $\alpha$ as follows. We write $Y(\tau,\sigma)$ for the imaginary part of $Z(\tau,\sigma)$ from \eqref{eq:Zoftausigma}.  
\begin{definition} Denote by $\mathcal{V}$ the set of Weyl chambers in $\GrK$, and let
 $\alpha_K$ be the map $\HU \rightarrow \GrK$ given by $(\tau, \sigma) \mapsto \R Y(\tau, \sigma)$. 
 A \emph{Weyl chamber $W$ of $\HU$} is a connected subset with $\alpha_K(W) = V \cap \alpha_K(\HU)$ for some $V \in \mathcal{V}$.
 If $V$ is of index $(n, p(\beta))$, with $n\in\Z$, $\beta \in L_0'/L$ we consider $W$ as being of this index, as well.
   
Also, to facilitate notation, we denote the subset of $\coneU$ corresponding to $W \subset \HU$ by by $W$, too.
\end{definition} 
The Weyl chambers of $\HU$ can also be described explicitly through inequalities. 
As usual, let $z = \ell' - \tau\hlf{\ell'}{\ell}\delta \ell + \sigma$, and set $Y(z) = p_K(\alpha(z))$, where $p_K$ is the projection from \eqref{eq:def_pK}. Note that $Y(z) = Y(\tau,\sigma)$. 
The Weyl chambers in $\GrK$ of index $(n, p(\beta))$,  for $n\in\Z$, $\beta \in L_0'/L$, consist precisely of the points not lying on any of the  divisors $\kappa^\perp$ for $\kappa \in p(\beta) + K$ with $\Qf{\kappa} = n$. Hence, $z$ is contained in one of the Weyl chambers with index $(n, p(\beta))$ if and only if $\blf{Y(z)}{\kappa} \neq 0$ for all $\kappa \in p(\beta) + K$, $\Qf{\kappa} =n$.

Since for $Y_L = \alpha(z)$, we have $\blf{Y_L}{\kappa} = \blf{Y}{\kappa}$, it follows that every Weyl chamber of index $(n,p(\beta))$ is defined through a system of inequalities of the form
\[
  s_\lambda \cdot\Im\frac{\hlf{z(\tau,\sigma)}{\kappa}}{\hlf{\ell'}{\ell}}>0, \quad\text{with}\quad s_\kappa \in \{ \pm 1 \},\quad\text{for $\kappa \in p(\beta) + K$, $\Qf{\kappa} = n$}. 
\]  

\paragraph{Weyl chamber condition} 
Conversely, let $W_{n,\beta}$ a Weyl chamber of index $(n, p(\beta))$. Then, for every $\kappa \in p(\beta) + K$, with $\Qf{\kappa} = n$, the sign of $\blf{Y(z)}{\kappa} = \blf{Y_L(z)}{\kappa}$ is constant on $W_{n,\beta}$. 

We introduce the following notational convention: Given a Weyl chamber $W$ in $\HU$ and a vector $\kappa \in K'$, write
\begin{equation}\label{eq:defWeylCond}
\Wpos{\kappa}{W}>0 \quad\text{if}\quad \Im\frac{\hlf{z}{\kappa}}{\hlf{\ell'}{\ell}} >0 \quad\text{for every $z=z(\tau,\sigma)\in W$}.
\end{equation}
In particular, this convention will be used in connection with Weyl chambers attached to the Fourier expansion of modular forms.

\paragraph{Weyl chambers attached to modular forms}
If $f$ is a weakly holomorphic modular form in $\Mweak_{1-m, \rho_L}$, with Fourier coefficients 
$c(n,\beta)$ in its principal part, the connected components of 
\[
\mathcal{G}_K  -  \bigcup\nolimits_{\beta \in L_0'/L} 
\bigcup\nolimits_{\substack{n\in\Z + \QfNop(\beta) \\ c(n,\beta) \neq 0 \\ n<0 }} H(n, p(\beta)) 
\]
are called the \emph{Weyl chambers} of $\mathcal{G}_K$ with respect to $f$, cf.\ \cite{Br02}, p.\ 88. Each such Weyl chamber can be written as an intersection of Weyl chambers of index $(n, p(\beta))$, where $\beta$ runs over $L_0'/L$, and $n$ over $\Z + \Qf{\beta}$ , with $c(n,\beta) \neq 0$.

Thus, in $\HU$, the Weyl chambers attached to $f$ are defined as the intersections of the form
\begin{equation}\label{eq:WeylChm_comp}
W = \bigcap\nolimits_{\beta \in L_0'/L} 
\bigcap\nolimits_{\substack{n\in\Z + \QfNop(\beta) \\ c(n,\beta) \neq 0 \\ n<0 }} W_{n,\beta},
\end{equation}
where $W_{n,\beta}$ denotes a Weyl chamber of index $(n,p(\beta))$ with $W\subset W_{n,\beta}$.

\section{The main theorem}\label{sec:mainthm}

\begin{theorem} \label{thm:BPmain}
Let $L$ be an even hermitian lattice of signature $(1,m)$, with $m\geq
1$, and $\ell \in L$ a primitive isotropic vector. Let $\ell' \in L'$ with $\hlf{\ell}{\ell'}
\neq 0$. Further assume that $\ell'$ is isotropic, as well.

Given a weakly holomorphic modular form $f\in \Mweak_{1-m, \rho_L}$
with Fourier coefficients $c(n, \beta)$ satisfying $c(n, \beta) \in
\Z$ for $n< 0$, there is a meromorphic
function $\Xi_f: \HU
\rightarrow \bar{\C}$ with the following properties: 
\begin{enumerate}
\item\label{item:MthmMf} $\Xi_f$ is an automorphic form of weight
$c(0,0)/2$ for $\GammaU{L}$, with some multiplier system $\chi$ of finite order. 
\item\label{item:MthmHd} The zeros and poles of $\Xi_f$ lie on Heegner
  divisors. The divisor of $\Xi_f$ on $\HU$ is given by
\[
\Div\bigl( \Xi_f \bigr) = \frac{1}{2} \sum_{\beta \in
  L'/L}\sum_{\substack{n \in \Z + \QfNop(\beta) \\ n<0}}c(n, \beta) 
\, \HeegU(n,\beta). 
\]
The multiplicities of the $\HeegU(n, \beta)$ are $2$, if 
$2\beta = 0$ in $L'/L$, and $1$ otherwise. Note that 
$c(n, \beta) = c(n, -\beta)$ and also that
$\HeegU(n,  \beta) = \HeegU(n, u\beta)$ for $u\in \OF^\times, \beta\neq 0$.
\item\label{item:MthmBp} For the cusp corresponding to
$\ell$ and for each Weyl chamber $W$, $\Xi_f(z)$ has an infinite product
expansion of the form
\[
 \begin{aligned}
 \Xi_f(z)  =  C\,&
 e\left(
\frac{\hlfa{z, \rho_f(W)}}{\hlf{\ell'}{\ell}}
\right)  \\ & \times 
 \prod_{\substack{\lambda \in K' \\
  \Wpos{\lambda}{W} > 0 }}
\prod_{\substack{\beta\in L_0'/L \\ p(\beta) = \lambda + K}}\!
\left[ 
1 - e\left(
\frac{\hlf{z}{\lambda}}{\hlf{\ell'}{\ell}} + 2\Re\left[ 
\frac{\bar{\zeta}\hlf{\beta}{\ell'}}{\bar{\delta}\hlf{\ell'} { \ell } }
 \right]
\right)
\right]^{c(\QfNop({\lambda}), \beta)}\\
\end{aligned}
\] 
where $z = z(\tau,\sigma) = \ell' - \delta\hlf{\ell'}{\ell}\tau \ell + \sigma$, the constant $C$ has absolute value $1$ and   
$\rho_f(W) \in K\otimes_\Z\Q$ is the Weyl vector attached to $W$.
The product converges normally for any $z$ lying in the complement of the
set of poles of $\Xi_f$, and satisfying $\hlf{z}{z}  >
4\lvert \hlf{\ell'}{\ell}\rvert^2 \abs{ n_0 }$, where $n_0
= min\lbrace n \in\Q;\, c(n, \beta) \neq  0\rbrace$. 
\item \label{item:MthmMl} The lifting is multiplicative: $\Xi(z; f + g) =
  \Xi(z; f) \cdot \Xi(z;g)$. 
\end{enumerate}
\end{theorem}
\begin{rmk}
The Weyl vector $\rho_f(W)$, attached to $W$ and $f$ can often be computed explicitly, see \cite{Bo98}, theorem 10.4. 
Also, if the Weyl chamber is decomposed as in \eqref{eq:WeylChm_comp}, then by a formula of Bruinier, cf.\ \cite{Br02}, p.\ 88,
  $\rho_f(W)$ is a linear combination of vectors $\rho_{n,\beta}$ attached to the $W_{n,\beta}$. 
\end{rmk}
\begin{corollary}\label{cor:MThUnimod}
We use the notation of the theorem. Suppose $L$ is the direct sum of a
hyperbolic plane $H \simeq \OF\oplus\DF^{-1}$ and a definite part $D$, with
$\hlf{D}{H} = 0$. Then, for a cusp corresponding to $\ell\in H$ and every
Weyl chamber $W$, the lift $\Xi_f(z)$ has an absolutely convergent Borcherds product expansion of the form
\begin{equation}\label{eq:maincor1}
\Xi_f(z) = C\,e\left( \frac{\hlfa{z , \rho_f(W)}}{\hlf{\ell'}{\ell}}\right)
\prod_{\substack{\lambda \in K' \\ \Wpos{\lambda}{W}>0}}
\left[1 - e\left( \frac{\hlf{z}{\lambda}}{\hlf{\ell'}{\ell}}\right)
\right]^{c(\QfNop(\lambda), \lambda)}.
\end{equation}
In particular, if $\hlf{\ell}{\ell'} = -\delta^{-1}$, the
Borcherds product of $\Xi_f(z)$ can be written as
\begin{equation}\label{eq:maincor2}
\Xi_f(z) = C\,e\left( \delta \hlf{z}{\rho_f(W)}\right)
\prod_{\substack{\lambda \in \Z\zeta\ell \oplus \Z\ell' \oplus D' \\
 \blf{\lambda}{W}>0}}
\Bigl( 1 - e\bigl( \delta {\hlf{z}{\lambda}} \bigr)
\Bigr)^{c(\QfNop(\lambda),\lambda)},
\end{equation}
where, as usual, $\zeta = \frac12 \delta$ if $D_\F$ is even and $\zeta=\frac12
(1 + \delta)$, if $D_\F$ is odd.
\end{corollary}
\begin{proof}[Proof of the corollary]
By assumption, $\ell' \in L$, as the hyperbolic part of $L$ is unimodular. Now, since $e_2$ is contained in $H\otimes_{\OF}\F$ and $e_2 \in L'$, one has $e_2 \in L$. With $\blf{e_2}{\ell} = 1$, we deduce that $\widthell = 1$. 
Hence, $L_0' = L'$. 

Further, $\left\lvert L'/L \right\rvert =  \left\lvert K'/K \right\rvert$, and for each $\lambda \in K'$ there is modulo $L$ only one representative $\beta$ with $p(\beta) = \lambda + K$. By remark \ref{rmk:beta_from_k}, it is given by 
$\beta = \lambda - \blf{\lambda}{f} \ell$. With $e_2=f$ and since $e_2 \perp K'$, we may identify $\beta$ with $\lambda = \lambda + 0 e_2 + 0 e_1 \in L'$. 
Finally, the $\ell$-component of $\lambda \in L'$ is given by $-\zeta \lambda_3$, where $\lambda_3$ denotes the $e_3$-component of $\lambda$. Since $\lambda_3$ is rational, we have $2\Re\left( -\delta^{-1} \bar\zeta \lambda_\ell\right)  = 0$.  Hence $\tr_{\F/\Q}(-\delta^{-1} \bar\zeta \beta_\ell) = 0 \pmod{\Z}$.

The second part, for $\hlf{\ell}{\ell'} = - \delta^{-1}$, follows with $e_4 = \ell'$ from \eqref{eq:basis_oddevenD}.    
\end{proof}
The proof of theorem \ref{thm:BPmain} is mostly based on results from Borcherds' seminal paper \cite{Bo98}, in which Borcherds uses a regularized theta-lift to construct a lifting from weakly holomorphic vector valued modular forms to automorphic forms on orthogonal groups. If $M$ is an even lattice of signature $(2,b)$, and $f\in \Mweak_{1-b/2, \rho_M}$, the lifting is given by the regularized integral 
\[
\Phi_M(Z,f) = \int^{reg}_{\mathcal{F}}\left\langle f(\tau),\Theta_M(\tau,Z)\right\rangle y^{b/2} \frac{dx\, dy}{y^2},
\]
with a (generalized) Siegel theta series $\Theta_M$. In this case, 
the additive lifting $\Phi_M(Z,f)$ can be used to obtain a multiplicative lifting 
$\Psi_M(Z,f)$. The functions $\Phi_M$ and $\Psi_M$ are related through
\begin{equation}\label{eq:logPsiL}
\log\left\lvert\ \Psi_M(Z,f)\right\rvert
= - \frac{\Phi_M(Z,f)}{4} - \frac12 c(0,0) \left( \log\left\lvert Y_M\right\rvert + \frac12 \Gamma'(1) + \log\sqrt{2\pi}\right),
\end{equation}
where $c(0,0)$ denotes the constant term in in the Fourier expansion of $f$.
Now for  $M=L$ and $b=2m$, we apply the machinery of the pullback under the embedding $\alpha$ developed in the previous sections to Borcherds' results, and  define
\begin{equation}\label{eq:defXi}
\Xi_f(z) \vcentcolon = \alpha^*\Psi_L(Z,f).
\end{equation}
\begin{proof}[Proof of the theorem]
The main reference for the proof is, of course, \cite{Bo98} in particular theorem 13.3 there, in which the properties of the multiplicative lift are formulated. 
%% Caveat (LaTex): Must use same item labels as in theorem!

\noindent\textit{1)} This follows from Borcherds' result by pullback, 
since $\Psi_L$ is meromorphic on $\HO$ and automorphic 
of weight $c(0,0)/2$ for (a subgroup of $\SO^+(L)$ containing) $\GammaO{L}$, cf.\ \cite{Bo98} lemma 13.1.
Then, $\Xi_f$ is a meromorphic function on $\HU$ transforming as a automorphic form of weight $c(0,0)/2$ for 
$\GammaU{L}$, since $\GammaU{L}$ is contained as a subgroup in $\GammaO{L}$.

That the multiplier system  $\chi$ has finite order, similarly follows by pullback. It is a consequence of a theorem by Margulis, see \cite{Br02}, p.\ 87, and \cite{Bo99Cor}. 

\noindent\textit{2)} 
The divisor of $\Psi_L(Z,f)$ is given by (cf.\ \cite{Br02}, theorem 2.22) 
\[
\Div(\Psi_L) = \frac12 \sum_{\beta \in L'/L}\sum_{\substack{n\in\Z + \QfNop(\beta) \\ n<0}} c(n,\beta) \HeegO(n,\beta).
\]
Now, by lemma \ref{lemma:Pb_Heeg_lambda} the restriction of each Heegner divisor occurring in the sum induces a Heegner 
divisor on $\HU$, with $\HeegO(n,\beta) \cap \alpha(\HU) = \HeegU(n,\beta)$. Thus, the divisor of the pullback $\Xi_f$ is given by
\[
\Div(\Xi_f) = \frac12 \sum_{\beta \in L'/L}\sum_{\substack{n\in\Z + \hlf{\beta}{\beta} \\ n<0}} c(n,\beta) \HeegU(n,\beta),
\]  
wherein, for $\beta \not\equiv 0 \mod{L}$, by \eqref{eq:Hulambda}, $\HeegU(n,\beta) = \HeegU(n,u\beta)$ for all $u\in\OF^\times$.

\noindent\textit{4)} This follows directly from the multiplicativity of the Borcherds lift $\Psi_L(Z,f)$.

\noindent\textit{3)} For the product expansion we must examine part of Borcherds' proof in more detail. For the Fourier expansion of the 
regularized theta-lift, Borcherds finds the following expression:
\begin{equation}\label{eq:PhiLafterEval}
\begin{aligned}
\Phi_L(Z,f)  & = 8\pi \blf{Y}{\rho_f(W)} 
+ c(0,0) \Bigl( \log(e_v^2) - \Gamma'(1) - \log(2\pi) \Bigr) \\
& - 2 \sum_{\substack{ \beta \in \Z \widthell \Z \\ \beta \neq 0 }}
 c(\beta \ell/\widthell,  0) 
 \log\left( 1 - e\left( \frac{\beta}{\widthell} \right)\right)
 \\ 
 &- 4 \sum_{\substack{\lambda \in K'  \\ \blf{W}{\lambda}>0 }}
 \sum_{\substack{\beta \in L_0'/L \\ p(\beta) = \lambda }} 
\sum_{k>0} c(\Qf{\lambda}, \beta)\,\cdot \frac{1}{k} 
e\Bigl( k \bigl(\blf{X}{\lambda}  + i  \left\lvert \blf{Y}{\lambda}\right\rvert + \blf{\beta}{e_2}\bigr) \Bigr),
\end{aligned}
\end{equation}
where $e_v$ denotes the projection of $e_1$ to $v= \R X_L \oplus \R Y_L \in \GrO$.
In order to determine the pullback of $\Phi_L$ under $\alpha$, we rewrite this 
expression in terms of $z$. We have
\[
\blf{X}{\lambda} = \Re\frac{\hlf{z}{\lambda}}{ \hlf{\ell'}{\ell}},\quad
\blf{Y}{\lambda} = \Im\frac{\hlf{z}{\lambda}}{ \hlf{\ell'}{\ell}}, \quad
e_v^2 
= \frac{\blf{X_L}{e}^2}{X_L^2}
= \frac{1}{Y^2} = \frac{2 \abs{\hlf{\ell}{\ell'}}^2}{\hlf{z}{z}},
\] 
also, recall  that $e_2 = \frac\zeta\delta \hlf{\ell'}{\ell}^{-1} \ell'$, as given in \eqref{eq:basis_oddevenD}.
Thus, the pullback of the first two lines from \eqref{eq:PhiLafterEval} is given by
\[
\begin{split}
 \;8\pi\cdot \Im\frac{\hlf{z}{\rho_f(W)}}{2\hlf{\ell'}{\ell}} 
+ c(0,0)  \cdot\biggl(- \log
 \frac{ \abs{\hlf{z}{z}} }{ 2|\hlf{\ell'}{\ell}|^2}
 -  \Gamma'(1) - \log(2\pi)\biggl) \\ 
  \;- \;2\sum_{\substack{\beta \in \Z/N\Z \\ \beta \neq 0 }}  
 c(0,\beta\ell/\widthell) \cdot \log\left(1 -  e\left(\frac{\beta}{\widthell}\right) 
  \right). 
\end{split}
\]
Note that $\blf{\rho_f(W)}{Y} = \blf{\rho_f(W)}{Y_L}$ as $\rho_f(W) \in K\otimes_\Z\R$ and $Y = p_K(Y_L)$.
The last line of \eqref{eq:PhiLafterEval} can be rewritten as 
\[
 2\sum_{\substack{\lambda \in K'  \\ \lambda\neq 0 }}
 \sum_{\substack{\beta \in L_0'/L \\ p(\beta) = \lambda }} 
 \sum_{k>0} \frac{ c(\Qf{\lambda}, \beta)}{k} \cdot\,
  e\!\left( k \left( \Re\frac{\hlf{z}{\lambda}}{\hlf{\ell'}{\ell}} + 
   i \left\lvert \Im \frac{\hlf{z}{\lambda}}{\hlf{\ell'}{\ell}}\right\rvert 
+ 2\Re \frac{ \bar\zeta\hlf{\beta}{\ell'} }{\bar\delta \hlf{\ell}{\ell'} }  
 \right)\right).
 \]
As for each $\lambda$ the expression 
$\blf{Y}{\lambda} = \Im\left( \hlf{z}{\lambda} \hlf{\ell'}{\ell}^{-1} \right)$
 is non-zero and has constant sign on $W$, we can restrict to $\lambda$ for which it is positive. 
Also, recall that $c(n,\beta) = c(n, -\beta)$. We obtain 
\begin{gather*}
 -4 \sum_{\substack{\lambda \in K'  \\ \Wpos{W}{\lambda}>0 }}
 \sum_{\substack{\beta \in L_0'/L \\ p(\beta) = \lambda }} 
 c(\Qf{\lambda}, \beta)\,\cdot
\log\left\lvert 1 - e\left( \frac{ \hlf{z}{\lambda}}{\hlf{\ell'}{\ell}}
+ 2\Re \frac{ \bar\zeta\hlf{\beta}{\ell'} }{\bar\delta \hlf{\ell}{\ell'} } 
 \right) \right\rvert.
\end{gather*}
After gathering all contributions, the pullback of $\Phi_L(f)$ takes the form
\[
\begin{split}
(\alpha^* \Phi_L(f))&(z) =  8\pi\cdot \Im\frac{\hlf{z}{\rho_f(W)}}{2\hlf{\ell'}{\ell}} 
+ c(0,0)  \cdot\biggl( - \log
 \frac{ \abs{\hlf{z}{z}} }{ 2|\hlf{\ell'}{\ell}|^2}
 -  \Gamma'(1) - \log(2\pi)\biggl) \\ 
  &- 2\sum_{\substack{\beta \in \Z/N\Z \\ \beta \neq 0 }}  
 c(0,\beta\ell/\widthell) \cdot \log\left(1 -  e\left(\frac{\beta}{\widthell}\right) 
  \right) \\
  & - 4\sum_{\substack{\lambda \in K'  
  \\ \Wpos{\lambda}{W}>0 }} 
 \sum_{\substack{\beta \in L_0'/L \\ p(\beta) = \lambda }}  
 c(\hlf{\lambda}{\lambda},  \beta)\cdot
 \log\left\lvert\, 1 -  
 e\left( \frac{\hlf{z}{\lambda}}{\hlf{\ell'}{\ell}} 
 + 2\Re\left[ \bar{\xi}
 \frac{\hlf{\beta}{\ell'}}{\hlf{\ell}{\ell'}}
\right] \right)\right\rvert.
\end{split}
\]
Now, with \eqref{eq:logPsiL} we can determine the logarithm of $\Xi_f$, as per definition $\Xi_f = \alpha^*\Psi_L$:
\[
\begin{aligned}
\log\lvert \Xi_f(z) \rvert & 
 =  - \frac{(\alpha^*\Phi_L(f))(z)}{4}
 - \frac{c(0,0)}{4}\left( 
 \log\frac{ \abs{\hlf{z}{z}} }{ 2|\hlf{\ell'}{\ell}|^2}
+ \Gamma'(1)  + \log(2\pi) \right).
\end{aligned}
\]
This results in the claimed form  of the product expansion.

The statement on convergence follows from a result of Bruinier, who in \cite{Br02} (theorem 3.22 on p.\ 88f)
gives a precise criterion for the normal convergence of the Borcherds product: The
 product converges normally on the complement of the set of poles if
 $\Qf{Y_L} = \Qf{Y}>\abs{n_0}$, with $n_0 = \min\lbrace n \in\Z;\, c(n, \gamma) \neq 0\rbrace$. 
 On $\HU$, with \eqref{eq:mapztoXLYL}, we thus have normal convergence if
    $\hlf{z}{z}> \abs{n_0}\abs{\hlf{\ell'}{\ell}}^2$, as claimed.  
\end{proof}
\begin{example}
Let $L = H$ be a hyperbolic plane as in example \ref{ex:hyp_plane}, 
and let $\ell = 1 \in \OF$ and $\ell' = -\delta^{-1} \in \DF^{-1}$.
Then, the Siegel domain model $\HU$ is just the usual upper half-plane $\Hp$.
Since $L$ is unimodular, the discriminant group is trivial, so $\GammaU{L} \simeq \Ug(L)$. 
Note that $\SL_2(\Z)\simeq\SU(L)$ is contained in $\GammaU{L}$. Also,
the Weil representation $\rho_L$ restricts to the usual multiplier system of $\SL_2(\Z)$ on $\C$.

For every $m\in \Z$ with $m>0$, there is a unique element $J_m$ of $\Mweak_0(\SL_2(\Z))$, with a $q$-expansion of the form $J_m = q^{-m} + \mathbf{O}(q)$.

The Weyl chambers attached to $J_m$ are stripe-shaped regions of $\Hp$, defined by inequalities in $\Im\tau$ and the \lq topmost\rq\ Weyl chamber is the half-plane $W = \{\tau \in\Hp\,;\, 2\Im\tau > m\abs{\delta}\}$. 
The Borcherds product expansion of $\Xi(\tau, J_{m})$ attached to $W$ is given by
\[
\Xi\bigl(\tau, J_m\bigr) = e\bigl( - \sigma(m) \tau \bigr)
\prod_{\substack{k,l\in\Z\\\Wpos{W}{(k,l)}>0}}
 \left( 1 - e\left( k\tau - l\bar\zeta \right)\right)^{c(kl)}, 
\]
where $\sigma(m) = \sum_{d\mid m} d$ is the divisor sum of $m$. Here, the Weyl chamber condition $\Wpos{W}{(k,l)}>0$ translates to  $2 k \Im\tau  + l\abs{\delta}>0$ for every $\tau$ in $W$. 
The product is absolutely convergent for $\Im\tau > 2m\abs{\delta}^{-1}$.
For a more detailed treatment, see \cite{Ho11}, chapter 5, and \cite{Ho13}.
\end{example} 

\section{Values on the boundary}\label{sec:limit}
In this section, we study the values taken by a Borcherds product $\Xi_f$ on the boundary points of $\HU$. 
As usual, we only consider the cusp at infinity. Also, in the following, we assume that the width of the cusp, $\widthell$, is equal to $1$.

\begin{theorem}\label{thm:limBopro}
We use the notation of theorem \ref{thm:BPmain}. 
Let $W$ be a Weyl chamber, such that the cusp corresponding to $\ell$ is contained in the closure of $W$. 
If this cusp is neither a pole nor a zero of $\Xi_f$, the limit $\lim_{\tau\rightarrow i\infty}\Xi_f(\tau, \sigma)$ is given by
\[
\lim_{\tau\rightarrow i\infty}\Xi_f(\tau, \sigma) = C e\bigl(\overline{\rho_{f}(W)}_\ell\bigr)
\prod_{\substack{\lambda \in K' \\ \lambda = \kappa\zeta \ell \\ \kappa \in \Q_{>0}}} 
\left(1 - e\bigl(\kappa\bar\zeta\bigr)\right)^{c(0,\beta)},
\] 
where $\rho_{f}(W)_\ell$ denotes the $\ell$-component of the Weyl vector $\rho_f(W)$, and $\beta$  is the lattice vector in $L'$, unique modulo $L$, for which $p(\beta+ L) = \lambda + K$.
\end{theorem} 

\begin{proof}
We denote the $\ell$- and $\ell'$-component of $\rho_f$ by $\rho_\ell$ and $\rho_{\ell'}$, respectively, and the definite part by $\rho_D \in D\otimes_{\OF}\C$.
Since $\rho_f \in K\otimes_\Z\Q$, we have $\rho_\ell \ell = \rho_3 e_3$, $\rho_{\ell'}\ell' = \rho_4 e_4$, where $\rho_3$ and $\rho_4$ denote the $e_3$- and $e_4$-components.
Similarly for a lattice vector $\lambda \in K'$, we write 
\[
\lambda = \lambda_\ell\ell +  \lambda_{\ell'}\ell' + \lambda_D = \lambda_3 e_3 + \lambda_4 e_4 + \lambda_D.
\]
Thus, with \eqref{eq:basis_oddevenD}, we have
\[
\begin{aligned}
\hlf{z}{\rho_f} & =  \bar\rho_\ell\hlf{\ell'}{\ell} - \tau\delta\lvert\hlf{\ell'}{\ell} \rvert^2 \bar\rho_{\ell'} + \hlf{\sigma}{\rho_D} \\
& = -\bar\zeta  \rho_3 \hlf{\ell'}{\ell} + \tau\rho_4 \hlf{\ell'}{\ell} + \hlf{\sigma}{\rho_D}.
\end{aligned}
\]
The Weyl vector term takes the form 
\[
e\left( \frac{\hlf{z}{\rho_f}}{\hlf{\ell'}{\ell}} \right) = \exp\left( 2\pi i \left[ - \bar\zeta \rho_3 + \tau \rho_4  + \frac{\hlf{\sigma}{\rho_D}}{\hlf{\ell'}{\ell}}  \right] \right). 
\]
Hence, $\Xi_f(\tau,\sigma)$ has a pole at infinity if $\rho_4<0$ and a zero if $\rho_4 > 0$. 

From now on, \emph{we assume $\lim_{\tau \rightarrow i\infty} \Xi_f(\tau,\sigma) \neq 0$ and $\neq \infty$}. In particular, $\rho_4 = 0$.  

Next, we claim that $\lambda_4$ is non-negative. To see this, consider the Weyl chamber condition $\Wpos{W}{\lambda}>0$ for $z$ with $\tau = it$, $t\gg 0$ and $\sigma$ of fixed norm. 
We have:
\[
\hlf{z}{\lambda} = - \bar\zeta \lambda_3 \hlf{\ell'}{\ell} +it \lambda_4\hlf{\ell'}{\ell} + \hlf{\sigma}{\lambda_D}.
\] 
Recall the definition of the Weyl chamber condition in \eqref{eq:defWeylCond}, clearly, if $t$ is large, $\Wpos{W}{\lambda}>0$ is satisfied only if $\lambda_4 \geq 0$. 
Since in the limit the corresponding factor in the product is trivial for $\lambda_4 >0$, we can restrict to $\lambda$ with $\lambda_4 = \lambda_{\ell'} = 0$. 

Since we assume $\widthell = 1$, by remark \ref{rmk:beta_from_k}, for each $\lambda$, there is modulo $L$ exactly one $\beta \in L'$ with $p(\beta) = \lambda + K$, given by $\beta = \lambda - \blf{\lambda}{f}\ell$, where $f$ is a vector in $L$ satisfying $\blf{\ell}{f} = 1$. Thus, the $\ell$-component of $\beta$ is given by $\lambda_\ell - \blf{\lambda}{f} = -\bar\zeta\lambda_3 - \blf{\lambda}{f}$. Note that $\blf{\lambda}{f}\in\Z$ since $f\in L$ and $\lambda \in L'$. Hence,
$2\Re\left( \bar{\delta}^{-1}\bar{\zeta} \beta_\ell \right)$ is an integer.

Up to here, we see that, in a suitable neighborhood of infinity, $\Xi_f$ can be written in the form
\begin{equation}\label{eq:BoPnearinfty}
\Xi_f(\tau,\sigma) = Ce\left(\bar \rho_{\ell} + \frac{\hlf{\sigma}{\rho_D}}{\hlf{\ell'}{\ell}}\right)
\prod_{\substack{\lambda \in K' \\ \lambda_{\ell'} = 0 \\ \Wpos{W}{\lambda}>0  }}
\left(1 - e\left( \bar{\lambda}_\ell + \frac{\hlf{\sigma}{\lambda_D}}{\hlf{\ell'}{\ell}}
 \right)\right)^{c(\hlf{\lambda_D}{\lambda_D}, \beta)}.
\end{equation}
As an automorphic form, $\Xi_f(\tau,\sigma)$ has a Fourier-Jacobi expansion of the form  
\begin{equation}\label{eq:jf_bopro}
\Xi_f(\tau,\sigma) = \sum_{n \geq 0} a_n(\sigma) e\left( \frac{n}{N} \tau \right).
\end{equation}
Note that by the non-vanishing assumption, 
$n \in \Z$ and by regularity, $n\geq0$. We have
\begin{equation} \label{eq:a0_isit}
a_0 =  
 \lim_{\tau\rightarrow i\infty} \Xi_f(\tau,\sigma).
\end{equation}
Beside the Fourier-Jacobi expansion \eqref{eq:jf_bopro}, the Borcherds product can also be rewritten as a series by expanding each factor as a binomial series and taking the resulting product. 
Thus, with the binomial series expansion, the right hand side of \eqref{eq:BoPnearinfty} becomes
\[
 C e\left(\bar \rho_{\ell} + \frac{\hlf{\sigma}{\rho_D}}{\hlf{\ell'}{\ell}}\right)\,\cdot\!\!
\prod_{\substack{\lambda \in K' \\ \lambda_{\ell'} = 0 \\ \Wpos{W}{\lambda}>0  }}
\sum_{n\geq 0} (-1)^n \binom{c(\hlf{\lambda_D}{\lambda_D},\beta)}{n } 
e\left( \bar{\lambda}_\ell + \frac{\hlf{\sigma}{\rho_D}}{\hlf{\ell'}{\ell}} \right)^n.
\] 
By multiplying all remaining factors, for the infinite product part we get
\[
1 + 
\sum_{k>0}\sum_{\substack{\lambda_1, \dotsc, \lambda_k \in K' \\ \Wpos{W}{\lambda_i}>0 \\ \lambda_{i,\ell'} = 0} }
\sum_{\substack{n_1, \dotsc, n_k \in \Z \\ n_i\geq 0}}
b\bigl( (\lambda_i, n_i)_{i=1,\dotsc, k} \bigr) 
e\left(\sum_{i=1}^k n_i \frac{\hlf{z}{\lambda_i}}{\hlf{\ell'}{\ell}} \right),
\]
with coefficients $b\bigl( (\lambda_i, n_i)_{i=1,\dotsc, k} \bigr)$ indexed by tuples of $k$ lattice vectors $\lambda_i$ 
and $k$ integers $n_i$. 

We set  $\tilde\lambda \vcentcolon= \sum_{i=1}^k n_i\lambda_i$. 
 By $\Z$-linearity,  $\tilde\lambda \in K'$ 
with $\tilde\lambda_{\ell'} =0$. 
Further, since the $\lambda_i$ satisfy the Weyl chamber condition $\Wpos{\lambda}{W}>0$ and  the $n_i$ are non-negative, each $\tilde\lambda$ satisfies $\Wposd{\tilde\lambda, W}\geq 0$. 
We gather all coefficients belonging to $\tilde\lambda$ and denote their sum as $B\bigl(\tilde\lambda\bigr)$. 
Now, comparing coefficients with \eqref{eq:a0_isit} gives
\begin{equation*} 
a_0 = 
C \biggl[ e\left(\bar\rho_\ell + \frac{\hlf{\sigma}{\rho_D}}{\hlf{\ell'}{\ell}} \right) 
+ e\bigl(\bar\rho_\ell \bigr) \!\sum_{\substack{\tilde\lambda \in K' 
\\  \Wposd{W ,\tilde\lambda} \geq 0}}
B\bigl(\tilde\lambda\bigr) e\left(\overline{\tilde\lambda_\ell} + \frac{\hlf{\sigma}{\lambda_D + \rho_D}}{\hlf{\ell'}{\ell}} \right) \biggr].
\end{equation*}
As the left hand side is constant, it follows that $\rho_D = 0$ and further that $\tilde\lambda_D = 0$ for all $\tilde\lambda$.
Whence $\lambda_D = 0$  for all those $\lambda$, which contribute non-trivial factors to the Borcherds product \eqref{eq:BoPnearinfty}.
Thus, re-inserting into the right-hand side of \eqref{eq:BoPnearinfty} we get
\[
 \lim_{\tau\rightarrow i\infty} \Xi_f(\tau,\sigma) = C e(\bar\rho_{\ell}) 
\prod_{\substack{\lambda = \lambda_\ell \ell \in K' \\ \Wposd{W, \lambda}>0} }
\left(1 - e(\bar\lambda_\ell)\right)^{c(0,\beta)}.
\]
 Since $\lambda  = \lambda_\ell \ell$ is contained in $K'$, it follows that $\lambda = \kappa e_3 = - \zeta\kappa \ell$, with a rational coefficient $\kappa$.
By the Weyl chamber condition,  $ 0 < \Im( - \bar\zeta\kappa ) =  \frac12 \abs{\delta} \kappa$. Thus, $\kappa >0$, as claimed.
\end{proof}
\begin{rmk*} 
A different approach for proving theorem \ref{thm:limBopro} is to consider the Borcherds lift $\Psi_L(Z,f)$ on a fixed one-dimensional boundary component.
 Evaluating this at the image $\alpha(\infty)$ of the point at infinity, see \eqref{eq:limit_alpha}, 
we can, due to lemma \ref{lemma:limit_embed},  obtain the value of $\lim_{\tau\rightarrow i\infty} \Xi_f(\tau,\sigma)$.
 This approach is used in the author's thesis \cite{Ho11}, chapter 4.3.
\end{rmk*}

\section{Modularity of Heegner divisors} \label{sec:heeg2}

In this section, we give an analogue of Borcherds' result on the modularity of Heegner divisors in the case of modular surfaces for orthogonal groups from \cite{Bo99}. A result similar to our theorem \ref{thm:AseriesMod} below, has been obtained independently by Liu in \cite{Liu}, using rather different methods.

Let $\C[L'/L][q^{-1}]$ be the space of Fourier polynomials (including constant terms) and $\C[L'/L][[q]]$ the space of formal power series. The residue-pairing is a non-degenerate bilinear pairing between these two spaces, which can be defined by putting
\[
\Res{f}{g} = \sum_{\substack{n\leq 0 \\ \beta \in L'/L}} c(n,\beta)b(-n,\beta),
\] 
for $f = \sum_{\beta, n\leq 0} c(n,\beta) q^n \in \C[L'/L][q^{-1}]$ and 
$g = \sum_{\beta, m\geq 0} b(m,\beta) q^m \in \C[L'/L][[q]]$. 
The space $\Mweak_{1-m,\rho_L}$ can be identified with a subspace of $ \C[L'/L][q^{-1}]$ by mapping a weakly holomorphic modular form to the non-positive part of its Fourier expansion.
 Likewise, the space $\mathcal{M}_{1+m, \rho_L^*}$ can be identified with a subspace of $\C[L'/L][[q]]$ by mapping a holomorphic modular form to its Fourier expansion.

Using Serre duality for vector-bundles on Riemann surfaces, in \cite{Bo99},   
Borcherds showed that the space  $\Mweak_{1-m,\rho_L}$ is the orthogonal complement of $\mathcal{M}_{1+m,\rho_L^*}$ with respect to the residue-paring $\Res{\,}{\,}$.
 Since the pairing is non-degenerate and $\mathcal{M}_{1+m,\rho_L^*}$ has finite dimension, $\mathcal{M}_{1+m,\rho_L^*}$ is also the orthogonal complement of $\mathcal{M}_{1-m,\rho_L}$. 
 In particular, the following holds (cf.\ \cite{Bo99}, theorem 3.1, or \cite{Br02} theorem 1.17).
\begin{lemma}\label{lem:exDual}
A formal power series $\sum_{\beta}\sum_{n>0}b(n,\beta)q^n \ebase_\beta \in \C[L'/L]\otimes\C[[q]]$ is the Fourier expansion of a modular form $g\in\mathcal{M}_{1+m,\rho_L^*}$ if and only if
\[
\sum_{\beta \in L'/L}\sum_{\substack{n\in \Z + \QfNop(\beta) \\ n\leq 0}} c(n,\beta) b(-n,\beta) = 0
\]
for every $f = \sum_{n, \beta} c(n,\beta)q^n \ebase_\beta \in \Mweak_{1-m, \rho_L}$.
\end{lemma}
By a result of McGraw, \cite{McGraw} theorem 5.6, the spaces $\Mweak_{1-m, \rho_L}$ and $\mathcal{M}_{1+m, \rho_L^*}$ have bases of modular forms with integer coefficients. 
Thus, a statement analogous to lemma \ref{lem:exDual} holds for power series and modular forms over $\Q$. Moreover, it suffices to check the vanishing condition for every $f$ with integral Fourier coefficients. 

Consider $\Chow(\UXGamma)$, 
the first Chow group of $\UXGamma$. Recall that $\Chow(\UXGamma)$ is isomorphic to the Picard group $\operatorname{Pic}(\UXGamma)$. 

Let $\pi: \tilde\UXGamma \rightarrow \UXGamma$ be a desingularization and denote by $\BDiv = \BDiv(\tilde\UXGamma)$ the group of boundary divisors of $\tilde\UXGamma$. We now consider a modified Chow group, the quotient $\Chow(\tilde\UXGamma)/ \BDiv$. Put $(\Chow(\tilde\UXGamma)/ \BDiv)_\Q = (\Chow(\tilde\UXGamma)/ \BDiv) \otimes_\Z \Q$.

Denote by $\Abdl{k}$ the sheaf of meromorphic automorphic forms on $\UXGamma$. 
By the theory of Baily-Borel, there is a positive integer $n(\Gamma)$, such that if $k$ is a positive integer divisible by $n(\Gamma)$, the sheaf $\Abdl{k}$ is an algebraic line bundle and thus defines an element in $\operatorname{Pic}(\UXGamma)$.
The pullback of $\Abdl{k}$ to $\tilde\UXGamma$ defines a class in $\Chow(\tilde\UXGamma)/\BDiv$, which we denote $c_1(\Abdl{k})$. 
More generally, if $k$ is rational, we choose an integer $n$ such that $nk$ is a positive integer divisible by $n(\Gamma)$ and put
$c_1(\Abdl{k}) = \frac{1}{n}c_1(\Abdl{nk}) \in (\Chow(\tilde\UXGamma)/\BDiv)_\Q$.

As the Heegner divisors are $\Q$-Cartier on $\UXGamma$, their pullbacks define elements in the modified Chow group $(\Chow(\tilde\UXGamma)/\BDiv)_\Q$.
\begin{theorem}\label{thm:AseriesMod} 
The generating series in $\Q[L'/L][[q]]\otimes (\Chow(\tilde\UXGamma)/\BDiv)_\Q$,
\[
A(\tau) = c_1(\Abdl{-1/2}) + \sum_{\beta \in L'/L} \sum_{\substack{n \in \Z + \QfNop(\beta) \\ n>0}} \pi^*\bigl(\HeegU(-n,\beta)\bigr)\, q^n \ebase_\beta 
\]
is a modular form in $\mathcal{M}_{1+m, \rho_L^*}$ with values in $(\Chow(\tilde\UXGamma)/\BDiv)_\Q$, i.e.\ $A(\tau)$ is contained in $\mathcal{M}_{1+m, \rho_L^*}\otimes (\Chow(\tilde\UXGamma)/\BDiv)_\Q$.
\end{theorem}
\begin{proof}
This follows from theorem \ref{thm:BPmain} and lemma \ref{lem:exDual}. 
Indeed, by lemma \ref{lem:exDual} it suffices  to show that 
\[
c(0,0)c_1(\Abdl{-1/2}) + \sum_{\beta \in L'/L} \sum_{\substack{n \in \Z + \QfNop(\beta) \\ n<0}} c(n,\beta)\, \pi^*\HeegU(n,\beta) = 0 \in (\Chow(\tilde\UXGamma)/\BDiv)_\Q,
\]
for every $f = \sum_{n,\beta} c(n,\beta)q^n \ebase_\beta$ in $\Mweak_{1-m,\rho_L}$ with integral Fourier coefficients. 
But this follows immediately from theorem \ref{thm:BPmain}, as the Borcherds lift $\Xi_f$ of $f$ is an automorphic form with divisor $\frac12 \sum_{n,\beta} c(n,\beta)\HeegU(n,\beta)$ of weight $c(0,0)/2$, i.e.\ up to torsion a rational section of $\Abdl{c(0,0)/2}$.
\end{proof}

\noindent\acknowledegements Much of the present paper is a short version of the author's thesis \cite{Ho11} completed at the TU Darmstadt.
 I am indebted to my thesis adviser, {J.\ Bruinier} for his insight and many helpful suggestions. 
I would also like to thank the anonymous referee, whose comments have helped to considerably improve this paper.

\bibliographystyle{hplain}
\bibliography{bouni2.bib}

\begin{thebibliography}{10}

\bibitem{BailyBorel}
W.~L. Baily, Jr. and A.~Borel.
\newblock Compactification of arithmetic quotients of bounded symmetric
  domains.
\newblock {\em Ann. of Math. (2)}, 84:442--528, 1966.

\bibitem{Bo95}
Richard~E. Borcherds.
\newblock Automorphic forms on {${\mathrm O}\sb {s+2,2}({\mathbf R})$} and
  infinite products.
\newblock {\em Invent. Math.}, 120(1):161--213, 1995.

\bibitem{Bo98}
Richard~E. Borcherds.
\newblock Automorphic forms with singularities on {G}rassmannians.
\newblock {\em Invent. Math.}, 132(3):491--562, 1998.

\bibitem{Bo99}
Richard~E. Borcherds.
\newblock The {G}ross-{K}ohnen-{Z}agier theorem in higher dimensions.
\newblock {\em Duke Math. J.}, 97(2):219--233, 1999.

\bibitem{Bo99Cor}
Richard~E. Borcherds.
\newblock Correction to: ``{T}he {G}ross-{K}ohnen-{Z}agier theorem in higher
  dimensions'' [{D}uke {M}ath. {J}. {\bf 97} (1999), no. 2, 219--233;
  {MR}1682249 (2000f:11052)].
\newblock {\em Duke Math. J.}, 105(1):183--184, 2000.

\bibitem{Br02}
Jan~H. Bruinier.
\newblock {\em Borcherds products on {O}(2, {$l$}) and {C}hern classes of
  {H}eegner divisors}, volume 1780 of {\em Lecture Notes in Mathematics}.
\newblock Springer-Verlag, Berlin, 2002.

\bibitem{BrFr}
Jan~H. Bruinier and Eberhard Freitag.
\newblock Local {B}orcherds products.
\newblock {\em Ann. Inst. Fourier (Grenoble)}, 51(1):1--26, 2001.

\bibitem{BHY13}
Jan~Hendrik Bruinier, Benjamin Howard, and Tonghai Yang.
\newblock Heights of {K}udla-{R}apoport divisors and derivatives of
  {L}-functions.
\newblock {\em ArXiv e-prints}, 2013, 1303.0549.

\bibitem{FrCub}
E.~Freitag.
\newblock Some modular forms related to cubic surfaces.
\newblock {\em Kyungpook Math. J.}, 43(3):433--462, 2003.

\bibitem{Ho11}
Eric Hofmann.
\newblock {\em Automorphic Products on Unitary Groups}.
\newblock PhD thesis, TU Darmstadt, 2011.
\newblock \url{http://tuprints.ulb.tu-darmstadt.de/2540/}.

\bibitem{Ho13}
Eric Hofmann.
\newblock Borcherds products for {${\mathrm{U}(1,1)}$}.
\newblock {\em ArXiv e-prints}, 2013, 1302.5301.

\bibitem{Liu}
Yifeng Liu.
\newblock Arithmetic theta lifting and {$L$}-derivatives for unitary groups,
  {I}.
\newblock {\em Algebra Number Theory}, 5(7):849--921, 2011.

\bibitem{McGraw}
William~J. McGraw.
\newblock The rationality of vector valued modular forms associated with the
  weil representation.
\newblock {\em Mathematische Annalen}, 326:105--122, 2003.

\end{thebibliography}

\end{document}